\documentclass{article}

\usepackage{babel}

\usepackage[a4paper,top=2cm,bottom=2cm,left=3cm,right=3cm,marginparwidth=1.75cm]{geometry}
\usepackage{nicematrix}
\usepackage{tikz}
\usepackage{amsmath, amssymb}
\usepackage{graphicx}
\usepackage[colorlinks=true, allcolors=blue]{hyperref}
\usepackage{amsfonts}
\usepackage{amsthm}

\newtheorem{theorem}{Theorem}[section]

\newtheorem{corollary}[theorem]{Corollary}
\newtheorem{lemma}[theorem]{Lemma}
\newtheorem{proposition}[theorem]{Proposition}
\newtheorem{definition}[theorem]{Definition}
\newtheorem{remark}[theorem]{Remark}
\newtheorem{example}[theorem]{Example}
\numberwithin{equation}{section}
\newcommand{\mycomment}[1]{}

\newcommand{\ZZ}{\mathbb Z}
\newcommand{\CC}{\mathbb C}
\newcommand{\PP}{\mathbb P}

\newcommand{\maxideal}{\mathfrak{m}}
\renewcommand{\AA}{\mathbb{A}}

\DeclareMathOperator{\im}{Im}
\DeclareMathOperator{\soc}{Soc}
\DeclareMathOperator{\CH}{CH}
\DeclareMathOperator{\Pic}{Pic}

\DeclareMathOperator{\grad}{grad}
\DeclareMathOperator{\Sym}{Sym}
\DeclareMathOperator{\adj}{adj}
\DeclareMathOperator{\Hilb}{Hilb}

\providecommand{\keywords}[1]
{
  	
  \textbf{\textit{Keywords:}} #1
}

\providecommand{\MSC}[1]
{
	
  \textbf{\textit{MSC Classes:}} #1
}

\title{Characteristic numbers of algebras}
\author{Jakub Jagiełła, Paweł Pielasa, Anatoli Shatsila}
\date{}

\begin{document}

\maketitle

\begin{abstract}
    We introduce \emph{characteristic numbers} of a finite commutative unital $\CC$-algebra, which are numerical invariants arising from algebraic intersection theory. We characterize Gorenstein and local complete intersection algebras in terms of their characteristic numbers. We compute characteristic numbers for certain families of algebras. We show that characteristic numbers are constant on $\Hilb_d(\AA^1)$, provide an explicit upper bound for characteristic numbers on the smoothable component of $\Hilb_d(\AA^n)$ and an explicit lower bound for characteristic numbers on the Gorenstein locus of $\Hilb_d(\AA^n)$ for $n \geq d-2$.
\end{abstract}

%\unnumberedfootnote{\textit{2020 Mathematics Subject Classification:} }
%\unnumberedfootnote{\textit{Key words and phrases:} }

\keywords{Gorenstein algebra, local complete intersection, algebraic intersection theory, variety of complete quadrics, tensor, characteristic number, Hilbert scheme of points}

\MSC{14C05, 14C17, 13H10, 13D10, 14E05}

\section{Introduction}
Finite algebras are ubiquitous in algebraic geometry and related fields, especially when studying moduli spaces via deformation theory, Hilbert schemes of points \cite{MR1839919}, secant varieties \cite{MR3092255, MR3121848}, algebraic topology \cite{MR3945742, MR4513165} and complexity theory \cite{landsberg_complexity, MR4432327}. The theory continues to develop; see \cite{MR4802586} for a modern survey of open problems. \\

Algebraic intersection theory has found significant applications in various areas of mathematics, particularly in the study of chromatic polynomials \cite{Huh1, Huh2, Huh3}. In his groundbreaking paper \cite{Huh1}, June Huh demonstrated that the coefficients of the reduced chromatic polynomial of a realizable matroid correspond (up to a sign) to certain intersection numbers of a subvariety associated with the matroid inside the permutohedral variety. This result was later extended to matroids realizable over a field of arbitrary characteristic, see \cite{Huh2}. The key insight came from identifying these coefficients with certain intersection numbers in toric geometry. This geometric perspective implies that the coefficients form a log-concave sequence. The log-concavity result was later generalized to non-realizable matroids in \cite{Huh3}, where Karim Adiprasito, June Huh, and Eric Katz showed, using combinatorial methods, that matroidal Chow rings are endowed with a Kähler package. \\

These geometric ideas were key ingredients of the framework developed in \cite{dinu2021applicationsintersectiontheorymaximum}. The aim of this paper is to use algebraic intersection theory to define new numerical invariants of finite algebras. Then, we characterize certain classes of algebras in terms of their characteristic numbers and calculate them for several families of examples. Moreover, we interpret some of these results in the language of Hilbert schemes of points.
%powtorzenie abstraktu po co?

\subsection{Characteristic numbers}
The characteristic numbers of algebras are defined using intersection theory on the variety of complete quadrics. They can be thought of as a generalization of characteristic numbers of matroids introduced by June Huh, giving a connection between geometry and combinatorics. In this section, we recall the historical context and applications of intersection theory on complete quadrics, introduce our notion of characteristic numbers, summarize some of the ideas present in the work of June Huh \cite{Huh1} and compare them with our setup.\\

Let $V$ be an $n$-dimensional $\CC$-vector space. The \emph{variety of complete quadrics} $\mathcal{CQ}(V)$ is the closure of the image of the rational map
\[
    \Phi_{\mathcal{Q}} \colon \mathbb{P}(S^2V) \dashrightarrow \mathbb{P}(S^2V) \times \mathbb{P}\left(S^2\left(\bigwedge^2 V\right)\right) \times \ldots \times \mathbb{P}\left(S^2\left(\bigwedge^{n-1} V\right)\right),
\]
sending the projectivization of $\varphi \in S^2 V \subset V \otimes V = \operatorname{Hom}(V^*, V)$ to 
the product of projectivizations of $\bigwedge^i \varphi \in S^2 (\bigwedge^i V) \subset  \bigwedge^i V \otimes \bigwedge^i V = \operatorname{Hom}(\bigwedge^i V^*, \bigwedge^i V)$.  
In coordinates, on the $i$-th factor, the map $\Phi$ sends a symmetric matrix to its $(i \times i)$-minors.
Alternatively, $\mathcal{CQ}(V)$ can be described as the Chow quotient of the Lagrangian Grassmannian via the $\CC^*$-action~\cite{K, Thaddeus}. We describe the construction in more detail in Section \ref{2.2}.

Let $L \subset S^2V$ be a $(d+1)$-dimensional linear subspace, containing a matrix of rank $n-1$ 
(so that $\Phi|_L$ is well-defined). The closure of the image of $\Phi|_{L}$ defines a $d$-dimensional 
subvariety $X_L$ of complete quadrics $\mathcal{CQ}(V)$.
For any tuple $(b_1, \dots, b_{n-1})$ of non-negative integers summing up to $d$, one can consider the intersection of $X_L$ with $b_i$ general hyperplanes on $\mathbb{P}(\bigwedge^{i}V \otimes \bigwedge^{i}V)$. By Bertini's theorem, it is a zero-dimensional variety, that is, a finite number of points. We denote its degree by $c^L_{b_1,\ldots,b_{n-1}}$ and call it the \emph{characteristic number} of $L$ indexed by $(b_1, \dots, b_{n-1}$). Alternatively, it can be interpreted as the product in the Chow ring of $\mathcal{CQ}(V)$:
\[
    c^L_{b_1, \dots, b_{n-1}} = [X_L] \cdot L_1^{b_1} \cdot \dots \cdot L_{n-1}^{b_{n-1}} \in \CH^n(\mathcal{CQ}(V)) \cong \ZZ,
\]
where $L_i$ are classes of pullbacks of hyperplanes in $\PP(S^2 (\bigwedge^i V))$.
This definition can be viewed a special case of the definition of characteristic numbers of a tensor in \cite[Definition 3.2]{CM}, as $L \subset S^2V$ can 
be identified with a subspace of partially symmetric tensors in $\CC^{d+1} \otimes S^2V \subset \CC^{d+1} \otimes V \otimes V$.

Consider a finite commutative unital $\CC$-algebra $A$ of dimension $n$. The multiplication 
$A \times A \to A$ is a symmetric billinear map corresponding to a tensor in $S^2 A^* \otimes A$. 
The image of the flattening of this tensor with respect to $A$ is an $n$-dimensional subspace of 
$S^2 A^*$. \emph{Characteristic numbers} of the algebra $A$ are the characteristic numbers of this 
subspace. We denote them by $c^A_{b_1, \dots, b_{n-1}}$. \\

The space of complete quadrics $\mathcal{CQ}(V)$ is a compactification of the space of non-degenerate quadrics in $\mathbb{P}^n$, structured to include degenerate cases. Nearly 150 years ago, Schubert published his influential work \cite{Schubert2} on quadrics satisfying various tangency conditions, such as passing through fixed points or being tangent to hyperplanes. Schubert calculus in this setting involves computing characteristic numbers that count such quadrics. The Schubert calculus for complete quadrics has since become a classical topic, with many remarkable results \cite{Semple1, Semple2, Tyrrell, Vainsencher}.

In the modern setting, questions in the enumerative geometry of quadrics can be addressed by computing characteristic numbers in the cohomology ring of $\mathcal{CQ}(V)$. Specifically, this cohomology ring has been described in terms of generators and relations, and algorithms have been developed to compute characteristic numbers \cite{CPI, CPII, LLT}.

The Schubert calculus for complete quadrics is closely linked \cite{MMW} to the maximum likelihood degree (ML-degree) of linear concentration models \cite{ML} and the algebraic degree of semidefinite programming \cite{DSP}. The ML-degree is a fundamental invariant in algebraic statistics that quantifies the complexity of a statistical model. It can be interpreted as counting the number of critical points of a likelihood function, which is of interest in statistics since the maximum likelihood estimate is one of them. 
% Furthermore, the ML-degree provides valuable information about the structure of algebraic varieties, see \cite{H} for a discussion on varieties with ML degree one.
The key connection arises in the characteristic number $c^L_{0,\ldots,0,d}$, which counts the 
number of non-degenerate quadrics in $n$ variables that pass through $n(n+1)/2 - (d+1)$ general points
and are tangent to $d$ general hyperplanes. This quantity coincides with the ML-degree $\phi(n, d+1)$ 
of the general linear concentration model associated with a $(d+1)$-dimensional subspace
$L \subset S^2V$ \cite[Proposition 2.4.1]{dinu2021applicationsintersectiontheorymaximum}.  
This insight has led to proofs of several significant conjectures in algebraic statistics, including the polynomiality of $\phi(n,d)$ in $n$ \cite{MMMSV} and an explicit formula for the ML-degree of a Gaussian graphical model defined by a cycle \cite{ADMV}. \\%For a detailed exposition on the connections to ML-degree computations, see \cite{HS}.\\

Let us describe the setup used by June Huh \cite{Huh3, Huh1, ALCO_2024__7_5_1479_0} and explain the connection to our work.
Let $V$ be an $(n+1)$-dimensional $\CC$-vector space. The \emph{permutohedral variety} $\Pi(V)$ is the closure of the image of the rational map 
$$
    \nabla F: \mathbb{P}(V) \dashrightarrow \mathbb{P}(V) \times \mathbb{P}\left(\bigwedge^2V\right) \times \ldots \times \mathbb{P}\left(\bigwedge^nV\right)
$$ 
given by 
$$
    [x_0:\ldots:x_n] \mapsto ([x_0:\ldots:x_n], [x_0x_1:x_0x_2:\ldots:x_{n-1}x_n],\ldots,[x_0\ldots x_{n-1}: \ldots : x_1\ldots x_n]).
$$ 
For a linear subspace $L \subset V$, the closure of the image of $\nabla F|_L$ defines a subvariety $W_L \subseteq \Pi(V)$ of the permutohedral variety.
In the original context, the subspace $L$ is associated to a realizable matroid $M$ and the 
corresponding subvariety $X_M$ is the wonderful compactification of the complement of the hyperplane arrangement associated to $M$. 
Defining characteristic numbers $c^L_{b_1, \dots, b_{n-1}}$ as above, and considering the characteristic sequence $c^L_{n, \dots, 0}, c^L_{n-1, \dots, 1},\ldots, c^L_{0, \dots, n}$ (which corresponds to the projective degree of the Cremona map), it is a remarkable result \cite{Huh1} that the characteristic sequence is (up to a sign) the sequence of coefficients of the reduced characteristic polynomial of $M$. This geometric insight played a key role in his proof of the long-standing Rota-Heron-Welsh conjecture, which asserts that these coefficients form a log-concave sequence. For a good expository text about this geometrical setting for matroids and the methods used in the proof of log-concavity, see \cite{H3}.

There is another, equivalent definition of the characteristic sequence. 
Given a subspace $L \subset V$, one can associate to it a subvariety of 
$\PP(V) \times \PP(\bigwedge^n V) = \PP(V) \times \PP(V^*)$ as above. The characteristic sequence of 
$L$ coincides with the sequence of coefficients by $[\PP^i \times \PP^{n-i}]$ in the representation 
of the class of this subvariety in the Chow ring of $\PP(V) \times \PP(V^*)$. Understanding 
these coefficients played a key role in \cite{H2}, where June Huh considers a complement $U$ of a 
hyperplane arrangement in $\CC^r$ and associates to it a subvariety in $\overline{U} \times \PP^n$.
 He proves that the coefficients of $[\PP^{r-i} \times \PP^{n-1-r+i}]$ are the alternating 
 coefficients of the Chern-Schwartz-MacPherson class of the characteristic function of $U$. 
 This allowed, among other things, to relate the Euler characteristic of $U$ to the computation 
 of the chromatic polynomial. In our work we use a similar interpretation in \ref{4.4} to prove explicit 
 formulas for characteristic number of algebras in $\Hilb_d(\mathbb{A}^1)$.

The permutohedral variety is a toric variety given by the Chow quotient of the product $(\PP^1)^n$ by the diagonal 
$\CC^*$-action. It can be described in a very concrete, combinatorial way. Given a projection of one 
polytope onto another, one can consider the associated \textit{fiber polytope}, which is a weighted Minkowski 
sum of fibers of this projection. In~\cite{BS} and~\cite{KSZ}, the authors provide a framework for 
understanding Chow quotients of toric varieties by toric actions in this language -- they showed that 
such a quotient is a toric variety described by an appropriate fiber polytope. In our setting, the 
Chow quotient of $(\PP^1)^n$ corresponds to a fiber polytope arising from projection of the 
$n$-dimensional hypercube $I^n$ onto the ``diagonal'' interval $I$.\\

Our use of the variety of complete quadrics to study invariants of subspaces of matrices can be 
seen as a generalization of the setup of the permutohedral variety used by June Huh 
(see \cite{CM, dinu2021applicationsintersectiontheorymaximum} for more details). One can identify 
$V$ with the subspace of $S^2 V$ consisting of diagonal matrices. Then, the map $\nabla F$ 
(composed with the embedding induced by $\bigwedge^i V \hookrightarrow S^2 \bigwedge^i V$) 
coincides with the restriction of the map $\Phi$ (defining the variety of complete quadrics)
to the subspace consisting of diagonal matrices. Thus, the subvariety of complete quadrics 
corresponding to diagonal matrices coincides with the permuthedral variety, considered as a 
subvariety of the complete quadrics. Note that the space of diagonal matrices can be thought of 
as the flattening of the multiplication tensor in the trivial algebra $\prod_{i=1}^n \CC$, 
so the permutohedral variety is just a special case of our main object of study, that is, 
the varieties associated to more general finite algebras. We believe that further understanding 
the intersection theory on $\mathcal{CQ}(V)$ can lead to interesting invariants defined in analogy 
to the ones in matroid theory.
This analogy can be further extended to the varieties $X_A$ we associate to finite algebras. 
In the original context, one associates a subvariety $W_L \subset \Pi(V) \subset \mathcal{CQ}(V)$ 
to a subspace of the space of diagonal matrices. In our work we associate a subvariety $X_L \subset \mathcal{CQ}(V)$ to a subspace of the space of symmetric matrices. The definition of characteristic numbers in our setting generalizes the notion of characteristic numbers of a linear subspace considered by June Huh. 
The interpretation of the permutohedral variety as the Chow quotient of product of projective lines also has its counterpart, as $\mathcal{CQ}(V)$ is a Chow quotient of Lagrangian Grassmannian. This viewpoint in both cases leads to the interpretation of $\mathbb{C}$-points of the variety as a sequence of broken orbits meeting on fixed components of the toric action (on $(\mathbb{P}
^1)^{n+1}$ or the Lagrangian Grassmannian respectively).

%\textcolor{red}{TODO Maybe add some of these above:} ~\cite{Huh1, Huh2, Huh3, HS, H, ADH, H2, H3} 

\subsection{Algebraic information encoded in characteristic numbers}

Characteristic numbers encode non-trivial information about the algebra. We obtained a few characterizations of Gorenstein algebras in terms of their characteristic numbers. The most interesting ones are given in the following proposition.
\begin{proposition}[Proposition \ref{p:gorenstein}]
    Let $A$ be a finite algebra. The following conditions are equivalent to $A$ being Gorenstein:
    \begin{enumerate}
        \item the characteristic sequence of $A$ is symmetric,
        \item $c^A_{0, \dots, 0, n-1} = 1$,
        \item $c^A_{0, \dots, 0, n-1} \neq 0$.
    \end{enumerate}
\end{proposition}

The characteristic sequence for a Gorenstein algebra shares many of the properties of other sequences 
of numerical invariants, such as the Hilbert function or Betti numbers. Nevertheless, there are 
important distinctions. It is known that both Hilbert function \cite{Stanley} and Betti 
numbers \cite{Boji} of Gorenstein graded algebras need not be unimodal. However, characteristic 
numbers of any Gorenstein algebra form a unimodal (in fact, even a log-concave) sequence. This follows directly from the Hodge-Teissier-Khovanskii inequality \cite[Lemma 33]{Huh1}, which says that if $H_1,\ldots, H_n$ are nef divisors on an $n$-dimensional variety $X$, then 
\[
    ([H_1]\cdot [H_1] \cdot [H_{3}] \cdot \ldots \cdot [H_n])([H_2]\cdot [H_2] \cdot [H_{3}] \cdot \ldots \cdot [H_n]) \leq ([H_1]\cdot [H_2] \cdot [H_{3}] \cdot \ldots \cdot [H_n])^2.
\]

Furthermore, we obtained a characterization of local complete intersections:
\begin{proposition}[Proposition \ref{p:ci}]
    Let $A$ be an $n$-dimensional local algebra. Then 
    \[
    		c^A_{0,n-1,0,\dots, 0} \leq 2^{n-1} - n
    \] 
    and the equality holds if and only if $A$ is a (local) complete intersection.
\end{proposition}

In both classes of algebras, there was a particular characteristic number determining a specific algebraic property. We hope that other characteristic numbers may also have a clear algebraic interpretation. 

\subsection{Explicit calculations and relation to the Hilbert schemes of points}

The \emph{Hilbert scheme} of $d$ points on the affine plane $\AA^n$ is a scheme parametrizing zero-dimensional closed subschemes $Z \subset \AA^n$ of length $d$, or equivalently, ideals $I \subset \CC[x_1, \dots, x_n]$ such that $\dim_\CC \CC[x_1, \dots, x_n] / I = d$. Formally, it can be defined as the scheme representing the functor 
\[
\begin{split}
    \Hilb_d(\AA^n) \colon \mathbf{Alg}_\CC &\to \mathbf{Set} \\
    A &\mapsto \{ I \subset A[x_1, \dots, x_n] \colon A[x_1, \dots, x_n]/ I \text{  is a locally free $A$-module of rank $d$}\} 
\end{split}
\]
In particular, the $\CC$-points of $\Hilb_d(\AA^n)$ are precisely $d$-dimensional $\CC$-algebras which are quotients of $\CC[x_1, \dots, x_n]$, that is
\[
    \Hilb_d(\AA^n)(\CC) = \{I \subset \CC[x_1, \dots, x_n] \colon \dim_\CC \CC[x_1, \dots, x_n] / I = d\}.
\]
We say that an algebra $A = \CC[x_1, \dots, x_n] / I$ is \emph{smooth} if $V(I) \subset \AA^n$ is 
smooth, that is, it consists of $d$ distinct reduced points. Each smooth algebra is isomorphic to 
$\prod_{k=1}^d \CC$. The locus of smooth points is open in its irreducible component in the Hilbert scheme. 
The \emph{smoothable component} of $\Hilb_d(\AA^n)$ is the closure of the locus of points corresponding to smooth algebras. Intuitively, it consists of length $d$ subschemes of $\AA^n$ which are limits of tuples of $d$ points. \\

The characteristic numbers of the smooth algebra $\prod_{k=1}^d \CC$ are known in the literature as mixed Eulerian numbers. They have numerous interpretations and are important objects of study in combinatorics, see \cite{10.1093/imrn/rnn153, mixedeuler, ALCO_2024__7_5_1479_0}. The main theorem of this paper is the following:
\begin{theorem}[Theorem \ref{t:numbers}]
    The characteristic numbers of the smooth algebra $\prod_{k=1}^d \CC$ and the characteristic numbers of the algebra $\CC[x]/(x^d)$ coincide.
\end{theorem}
The main idea of the proof is to show that the recurrence relations which determine mixed Eulerian numbers hold for the characteristic numbers of the algebra $\mathbb{C}[x]/(x^d)$. This is done by a careful analysis of the intersection of the associated subvariety of complete quadrics (interpreted as a Chow quotient of the Lagrangian Grassmannian) with two types of divisors in the Chow ring of complete quadrics. 

This has an interesting consequence from the perspective of Hilbert schemes of points. It is well-known that each point of $\Hilb_d(\AA^1)$ is a degeneration of a smooth algebra (treated as a point in $\Hilb_d(\AA^1)$ via e.g.~$\prod_{k=1}^d \CC \cong \CC[x]/(x^d - 1)$) and degenerates to $\CC[x]/(x^d)$. Characteristic numbers can only drop under degeneration, which leads to the following corollary:
\begin{corollary}\label{cor:hilb_constant}
    The characteristic numbers are constant on $\Hilb_d(\AA^1)$.
\end{corollary}

One can intuitively think of this result in the following way.
We can consider how the subvariety of complete quadrics corresponding to the multiplication tensor of an algebra changes under degeneration. The result above allows us to better understand how those subvarieties behave for families $\AA^1 \to \Hilb_d(\AA^1)$. Namely, as we vary the
algebra, the subvariety does not break into multiple irreducible components which would give non-zero contributions to the intersection numbers. This shows that algebras in $\Hilb_d(\AA^1)$ behave especially well in families, as already for $\Hilb_d(\AA^2)$ there are numerous examples of the intersection numbers dropping.

We have the following corollary:
\begin{corollary}\label{cor:hilb_bounded}
    The characteristic numbers on the smoothable component of $\Hilb_d(\AA^n)$ are bounded from above by the characteristic numbers of the smooth algebra, that is, the mixed Eulerian numbers.
\end{corollary}
In general, the smoothable component need not be the whole Hilbert scheme, and the geometry of other components is poorly understood. This corollary could potentially give a new method of detecting points outside the smoothable component. \\

The \emph{big Coppersmith-Winograd tensor} is a minimal border rank tensor important in complexity theory~\cite[Section 5.5]{landsberg_complexity}. It is the main building block of the algorithms proving to best known bounds on the exponent of matrix multiplication $\omega$, see~\cite{MR4262465}. It arises as the multiplication tensor of the algebra
\[
    CW_n = \frac{\CC[x_0, x_1, \dots, x_{n-2}]}{(x_i x_j \colon i \neq j) + (x_i^2 - x_j^2 \colon i \neq j) + (x_0^3)}.
\]
We calculate its characteristic numbers in Subsection \ref{s:cw}. 

This has an interesting interpretation in the language of Hilbert schemes. The algebra $CW_n$ is Gorenstein because it is apolar to the form $\alpha_0^2 + \alpha_1^2 + \dots + \alpha_{n-2}^2$. A recent result \cite[Proposition 4.1]{MR4513165} shows that each Gorenstein algebra of dimension $n+1$ degenerates to $CW_n$, which gives a bound for characteristic numbers of such algebras. 
\begin{corollary}\label{cor:hilb_gorenstein}
    The characteristic numbers of the Gorenstein locus of $\Hilb_d(\AA^n)$, where $n \geq d-2$, are bounded from below by the characteristic numbers of $CW_{d-1}$.
\end{corollary}

\subsection{Acknowledgements}

We would like to express our deepest gratitude to Mateusz Michałek for his continuous support during our stay at the University of Konstanz. We thank Viktoriia Borovik, Matěj Doležálek, Joachim Jelisiejew, Joseph Landsberg, and Julian Weigert for their helpful suggestions to improve earlier drafts. We thank Austin Conner for sharing his code. The second author thanks the Institute of Advanced Study for the hospitality during his visit as part of the Special Year on Algebraic and Geometric Combinatorics. During the preparation of the article, Jakub Jagiełła was supported by the DAAD/Ostpartnerschaften programme and National Science Centre grant 2023/50/E/ST1/00336, Paweł Pielasa was supported by the DAAD/Ostpartnerschaften programme and Anatoli Shatsila was supported by the Polish National Science Center grant 2024/53/N/ST1/01634.

\tableofcontents

\section{Preliminaries}\label{s:prelim}

\subsection{Algebras}
In this section we gather some basic definitions and results about algebras considered in this paper. \\

We always use the term \emph{algebra} to refer to the following:
\begin{definition}
    An algebra is an associative unital commutative finite $\CC$-algebra.
\end{definition}
Each  algebra can be represented as $A = S/I$, where $S := \CC[x_1, \dots, x_m]$ is a polynomial ring in finitely many variables, $I \subset S$ is an ideal and $\dim_\CC A < \infty$. If $A = S/I$, then the condition $\dim_\CC A < \infty$ is equivalent to $A$ being a ring of Krull dimension zero, that is, each prime ideal of $A$ is maximal. \\
Each algebra $A$ decomposes as a finite direct sum of local algebras (namely, there exist maximal 
ideals $\maxideal_1, \dots, \maxideal_r \subset A$ such that the direct sum of localization maps 
$A \to \prod_{i=1}^r A_{\maxideal_i}$ is an isomorphism). We usually restrict our attention 
to the local case and then retrieve the general case. This additional assumption allows the use of more tools from commutative algebra. \\

Now we recall definitions of several classes of algebras considered in this paper.
\begin{definition}
    An algebra $A$ is Gorenstein if there exists a linear functional $\lambda\colon A \to \mathbb{C}$ such that 
    \[
        A \times A \xrightarrow[]{\cdot} A \xrightarrow[]{\lambda} \mathbb{C}
    \]
    is a perfect pairing.   
\end{definition}
The algebra $A$ is Gorenstein if each of the local algebras $A_{\maxideal_i}$ is Gorenstein. If $(A, \maxideal)$ is local, then the condition from the definition above is equivalent to the condition that the socle $\soc(A) := (0 : \maxideal) = \{a \in A \: | \: \maxideal a = 0\}$ is 1-dimensional. %TODO maybe reference for this fact?
In this case $\lambda$ can be chosen to be the dual (after fixing a $\CC$-basis of $A$) of any non-zero element of $\soc(A)$. 

\begin{definition}
    An algebra $A$ is a complete intersection if there exists a regular ring $R$ and a regular 
    sequence $f_1, \dots, f_n \in R$ such that $A \cong R/(f_1, \dots, f_n)$. An algebra $A$ is locally a complete intersection if each localization $A_\maxideal$ is a complete intersection.    
\end{definition}
This turns out to be independent of the choice of the regular ring $R$, so we can always pick $R = S$. In the local case being a regular sequence is independent of the order of elements of this sequence.

Each local complete intersection algebra is Gorenstein. The other implication does not hold, which is illustrated by the following example.
\begin{example}
    Consider the algebra 
    \[
        A = \CC[x,y,z]/(x^2, y^2, xz, yz, z^2 - xy).
    \]
    This is a local algebra, with the maximal ideal given by $\maxideal = (x, y, z)$. An explicit calculation shows that $\soc(A) = \CC z^2$, so $A$ is indeed Gorenstein. However, the ideal $I = (x^2, y^2, xz, yz, z^2 - xy)$ cannot be generated by less than 5 elements, so $A$ is not a complete intersection. %Note that $I$ is not regular because $xz$ is a zero-divisor in $\mathbb{C}[x,y,z]/(x^2,y^2)$.
\end{example}

\subsection{Characteristic numbers}
\label{2.2}
In this section we define the characteristic numbers of algebras and introduce prerequisite notions. We start by introducing the variety of complete collineations, following \cite{dinu2021applicationsintersectiontheorymaximum}. Then we use it to define characteristic numbers of a linear subspace of the space of square matrices. The characteristic numbers of an algebra are then given by characteristic numbers of a certain linear subspace associated to this algebra.

\begin{definition}
Let $V$ be a $n$-dimensional vector space and consider the rational map
\[
    \Phi_{\mathcal{C}}: \mathbb{P}(V \otimes V) \dashrightarrow \mathbb{P}(V \otimes V) \times \mathbb{P}\left(\bigwedge^2V \otimes \bigwedge^2V\right) \times \ldots \times \mathbb{P}\left(\bigwedge^{n-1}V \otimes \bigwedge^{n-1}V\right),
\]
sending the class of $\varphi \in V \otimes V = \operatorname{Hom}(V^*, V)$ to the classes of $\bigwedge^i \varphi \in  \bigwedge^iV \otimes \bigwedge^i V = \operatorname{Hom}(\bigwedge^i V^*, \bigwedge^i V)$.
The closure of the image of this map is called the variety of complete collineations and is denoted by $\mathcal{CC}(V)$.    
\end{definition}

After choosing a basis on $V$ and taking the dual basis on $V^*$, we can describe $\Phi_{\mathcal{C}}$ 
very concretely. The elements of $V \otimes V$ are matrices $M$ of the corresponding linear maps 
$V^* \to V$. In these coordinates, $\bigwedge^i M$ is an $\binom{n}{i} \times \binom{n}{i}$-matrix 
whose entries are the $(i\times i)$-minors of $M$.

Alternatively, $\mathcal{CC}(V)$ can be obtained as a sequence of blowups
\[
    \PP(V \otimes V) = X_0 \leftarrow X_1 \leftarrow \dots \leftarrow X_{n-2} = \mathcal{CC}(V),
\]
where $X_i$ is the blow-up of $X_{i-1}$ at the strict transform of the locus of rank $i$ matrices.

\begin{definition}
\label{cq}
    Let $V$ be an $n$-dimensional vector space and consider the rational map
    \[
        \Phi_{\mathcal{Q}}: \mathbb{P}(S^2V) \dashrightarrow \mathbb{P}(S^2V) \times \mathbb{P}\left(S^2\left(\bigwedge^2V\right)\right) \times \ldots \times \mathbb{P}\left(S^2\left(\bigwedge^{n-1}V\right)\right)
    \]
    sending the class of $\varphi \in S^2 V \subset V \otimes V = \operatorname{Hom}(V^*, V)$ to the 
    classes of $\bigwedge^i \varphi \in   S^2 (\bigwedge^iV) \subset \bigwedge^iV \otimes \bigwedge^iV = \operatorname{Hom}(\bigwedge^i V^*, \bigwedge^i V)$.
The closure of the image of this map is called the variety of complete quadrics and is denoted by $\mathcal{CQ}(V)$.    
\end{definition}

We can interpret elements of $S^2 V$ as symmetric matrices and describe this map explicitly, analogously as above. Note that $\mathcal{CQ}(V)$ can be also described as 
\[
    \mathcal{CQ}(V) = \mathcal{CC}(V) \cap \mathbb{P}(S^2V) \times \mathbb{P}\left(S^2\left(\bigwedge^2V\right)\right) \times \ldots \times \mathbb{P}\left(S^2\left(\bigwedge^{n-1}V\right)\right).
\]

The variety $\mathcal{CQ}(V)$ can be also interpreted as a Chow quotient of the Lagrangian Grassmannian. We briefly describe the construction here, for details see \cite{Thaddeus}.
Consider the Grassmannian $\mathrm{Gr}(n, V \oplus V^*)$. Let $\omega$ be the standard symplectic form on $V \oplus V^*$. 
\begin{definition}
    The Lagrangian Grassmannian $\mathcal{LG}(V)$ is the subvariety of the Grassmannian $\mathrm{Gr}(n, V \oplus V^*)$ consisting of subspaces $W$ such that $\omega(W,W) = 0$. 
\end{definition}
The Grassmannian $\mathrm{Gr}(n, V \oplus V^*)$ admits an action of the torus $\CC^*$ with weight $1$ 
on $V$ and weight $-1$ on $V^*$. The symplectic form $\omega$ is invariant under this action, 
so this action descends to an action on the Lagrangian Grassmannian $\mathcal{LG}(V)$. Thus, one can consider the Chow quotient of $\mathcal{LG}(V)$ by the $\CC^*$-action. It is constructed as follows.

There exists an open subset $U \subset \mathcal{LG}(V)$ such that for each $x \in U$, the orbit 
closure $\overline{\CC^*\cdot x}$ represents the same homology class in $H_2(\mathcal{LG}(V), \ZZ)$, and the 
usual geometric quotient $U / \CC^*$ exists. There exists
a morphism taking an orbit $\CC^* \cdot x$ to the associated $1$-cycle $\overline{\CC^*\cdot x}$, 
considered as a point of the Chow variety parametrizing $1$-cycles. The Chow quotient of 
$\mathcal{LG}(V)$ is defined to be the closure of the image of $U / \CC^*$ in the Chow variety. 
Instead of trying to form a quotient by directly identifying orbits (which may lead to difficulties 
when orbits are not closed or when the quotient is not well-behaved), the Chow quotient “remembers” each 
orbit by considering its closure as an algebraic cycle. The Chow variety is a parameter space for 
such cycles. Thus, the Chow quotient encodes the geometric information of the orbit in a well-behaved 
projective variety. 

The Chow quotient of $\mathcal{LG}(V)$ is isomorphic to $\mathcal{CQ}(V)$. For $i=0,\dots,n$, there exists a closed embedding $\iota_i:\mathrm{Gr}(i, V) \hookrightarrow \mathcal{LG}(V)$ given by $U \mapsto U \oplus U^\perp$, where $U^\perp$ denotes the orthogonal complement of $U$ with respect to $\omega$. The construction of $\mathcal{CQ}(V)$ as the Chow quotient of $\mathcal{LG}(V)$ enables us to identify $\CC$-points of $\mathcal{CQ}(V)$ with broken orbits of the $\CC^*$-action on $\mathcal{LG}(V)$, that is, tuples of points 
from different embedded Grassmannians $\mathrm{Gr}(i,V)$ and $1$-dimensional orbits connecting the consecutive points. 
For a subspace $U \in \mathcal{LQ}(V)$, we have that 
\[
\lim_{t \to \infty} t \cdot U = \mathrm{pr}_{V}(U) \oplus U \cap V^*,\quad  \lim_{t \to 0} t \cdot U = U \cap V \oplus \mathrm{pr}_{V^*}(U).
\]
Thus for $k<l$ and subspaces $W_1 \in \mathrm{Gr}(k, V), W_2\in \mathrm{Gr}(l,V), W_1\subset W_2$ the unbroken orbits between
\[ W_1 \oplus W_1^{\bot} \in \mathcal{LG}(V) \text{ and } W_2 \oplus W_2^{\bot} \in \mathcal{LG}(V)\] 
are parametrized by the equivalence classes of isotropic subspaces $U \subset V \oplus V^*$ with the property $W_1 \oplus W_2^{\bot} \subseteq U \subseteq W_2 \oplus W_1^{\bot}$. The closure of the locus of such orbits is given by the variety $\mathcal{CQ}(W_2/ W_1)$.
We use this interpretation in Subsection \ref{s:smooth}.\\

Consider the coordinate projections
\[\pi_i: \mathcal{CQ}(V) \to \mathbb{P}\left(S^2\left(\bigwedge^iV\right)\right)\ \text{ and }\ \pi_i: \mathcal{CC}(V) \to \mathbb{P}\left(\bigwedge^iV \otimes \bigwedge^iV\right) \]
for $i = 1,\ldots,n-1$.
\begin{definition}
    For $i = 1,\ldots,n-1$ let $H_i$ denote a hyperplane in $\mathbb{P}\big(S^2(\bigwedge^iV)\big)$. We define $L_i \in \Pic_{\mathbb{Q}}(\mathcal{CQ}(V))$ to be the divisor $\pi_i^*([H_i])$.
\end{definition}
Analogously we can define the divisors $L_i$ on $\mathcal{CC}(V)$ to be the pullbacks of hyperplanes $H_i$ in $\mathbb{P}(\bigwedge^iV \otimes \bigwedge^iV)$. Their classes form bases of the vector spaces $\Pic_{\mathbb{Q}}(\mathcal{CQ}(V))$, $\Pic_{\mathbb{Q}}(\mathcal{CC}(V))$. The divisors $L_i$ are of fundamental importance in classical enumerative geometry. The intersection product $L_1^{a_1}\cdot L_2^{a_2}\cdot \dots L_{n-1}^{a_{n-1}}$ in $\CH(\mathcal{CQ}(V))$, with $\sum\limits_{i=1}^{n-1} a_i = \frac{n\cdot (n+1)}{2}-1$, counts the number of quadrics containing $a_1$ points, tangent to $a_2$ lines, $a_3$ planes, etc. This number is always bounded by the Bézout bound, and the difference corresponds to the multiplicity of the intersection of general hypersurfaces of appropriate degree at the base locus. This interpretation will play a role in our proof of Proposition \ref{p:ci}.\\

Consider a linear subspace $L \subset S^2V$ containing a matrix of rank at least $n-1$. We impose this assumption on $L$ so that the map $\Phi_{\mathcal{C}}|_{\mathbb{P}(L)}$ is well-defined. The main objects of study in this paper are the following varieties:
\begin{definition}
    Let $U_L \subset \mathbb{P}(L)$ be the nonempty open subset of $L$ corresponding to matrices of rank at least $n-1$. We define $X_L$ to be the closure of the image of $U_L$ under the rational map 
\[
    \Phi_{\mathcal{Q}}|_{U_L} : U_L \to \mathbb{P}(S^2V) \times \mathbb{P}\left(S^2\left(\bigwedge^2 V\right) \right) \times \ldots \times \mathbb{P}\left(S^2\left(\bigwedge^{n-1}V \right)\right).
\]
\end{definition}
Equivalently, it is the strict transform of $\mathbb{P}(L)$ under the projection $\pi_1 : \mathcal{CQ}(V) \rightarrow \mathbb{P}(S^2V)$.
\begin{lemma}\label{l:irreducible}
    The variety $X_L$ defined above is irreducible.
\end{lemma}
\begin{proof}
    Let $U_L$ be a dense open, and hence irreducible, subset of $\mathbb{P}(L)$ where $\Phi_{\mathcal{Q}}|_{U_L}$ is a well-defined morphism.
    Then, since the restriction $\Phi_{\mathcal{Q}}|_{U_L} \colon U_L \rightarrow \Phi_{\mathcal{Q}}(U_L)$ is an isomorphism, $\Phi_{\mathcal{Q}}(U_L)$
    is irreducible. Thus the closure $X_L = \overline{\Phi_{\mathcal{Q}}(U_L)}$ is also irreducible.
\end{proof}
 
Let $d+1 := \dim_\CC L$. Then the variety $X_L$ defined above has dimension $d$. We define the polynomial
\[
    P_L(D) \colon \Pic_{\mathbb{Q}}(\mathcal{CQ}(V))  \ni D \mapsto  [X_L]\cdot D^d \in \CH^n(\mathcal{CQ}(V)) \cong \ZZ,
\]
where $\CH^n(\mathcal{CQ}(V))$ denotes the highest degree summand of the Chow ring of $\mathcal{CQ}(V)$, classifying $0$-cycles up to rational equivalence.
In coordinates corresponding to the basis $\{L_i\}$ of $\Pic_{\mathbb{Q}}(\mathcal{CQ}(V))$, we have 
\[
    P_L\left(\sum_{i=1}^{n-1}t_iL_i\right) = \sum_{b_1 + \ldots + b_{n-1} = d} \binom{d}{b_1,\ldots,b_{n-1}} \cdot [S] \cdot L_1^{b_1} \ldots L_{n-1}^{b_{n-1}} \cdot  t_1^{b_1} \ldots t_{n-1}^{b_{n-1}}.
\]
\begin{definition}
    The coefficients $c^L_{b_1,\ldots,b_{n-1}} := [X_L] \cdot \prod_{i=1}^{n-1}L_i^{b_i}$ are integers called the characteristic numbers of the subspace $L$. The sequence of characteristic numbers of $L$ indexed by $(d-i, 0, \dots, 0, i)$ for $i = 0, 1, \dots, n$ is called the characteristic sequence of $L$. 
\end{definition}
We often drop $L$ from the notation if it is clear from the context. \\

The characteristic numbers defined above can be described in the following way. Choose a tuple $(b_1, \dots, b_{n-1})$ of non-negative integers such that $b_1 + \dots + b_{n-1} = d$. Then, by Bertini's theorem, the intersection of $X_L$ with $b_i$ general hyperplanes in $\mathbb{P}(S^2(\bigwedge^{i}V) )$ consists of finitely many reduced points. By definition, this number of points coincides with the product $[X_L] \cdot L_1^{b_1} \cdot \dots \cdot L_{n-1}^{b_{n-1}}$ in the Chow ring $\mathcal{CQ}(V)$, that is with $c^L_{b_1,\ldots,b_{n-1}}$.\\
%\begin{definition}
%    Let $L \subset S^2V$ be a $(d+1)$-dimensional subspace containing a matrix of rank $\geq n-1$
%    and let $(b_1, \dots, b_{n-1})$ be a tuple of non-negative integers summing up to $d$. Then we define the characterstic number $c^L_{b_1, \dots, b_{n-1}}$ to be the intersection product
%    \[
%    c^L_{b_1, \dots, b_{n-1}} = [X_L] \cdot L_1^{b_1} \cdot \dots \cdot L_{n-1}^{b_{n-1}} \in \CH^n(\mathcal{CQ}(V)) \cong \ZZ,
%    \]
%    in the Chow ring $\CH(\mathcal{CQ}(V))$, where $L_i$ are classes of pullbacks of hyperplanes in $\PP(S^2 (\bigwedge^i V))$.
%\end{definition}

If some of the $b_i$ are zero, then we can consider the image of $X_L$ in a smaller product of $\mathbb{P}(S^2(\bigwedge^{i}V))$ and conduct our computations there.

\begin{example}\label{e:grad}
    The characteristic number $c^L_{0, \dots, 0, d}$ is the degree of the restriction of rational map $ \varphi \colon \mathbb{P}(S^2V) \to \mathbb{P}(S^2 (\bigwedge^{n-1}V))$ to the subvariety $\PP(L)$. The map $\varphi$ coincides with the projectivization of the gradient of the determinant $\det \in \bigwedge^nS^2V^*$ 
    \[
        \grad \det \colon \PP(S^2 V) \dashrightarrow \PP (S^2V^*) \cong \mathbb{P}\left(S^2\left(\bigwedge^{n-1}V\right)\right).
    \]
    After further restricting to the open subset of invertible matrices $U_L \subset \mathbb{P}(L)$, this is just the projectivization of the map $A \mapsto A^{-1}$.
\end{example}

\begin{remark}\label{r:generalization}
    If the linear subspace contains matrices of ranks not exceeding $k < n-1$, then the map $\Phi_{\mathcal Q}$ is not well-defined.
    However, we can still talk about intersection numbers indexed by sequences $(b_1, \dots, b_k, 0, \dots, 0)$, by considering $\Phi_{\mathcal Q}$ truncated to the first $k$ factors of the codomain instead of the full map $\Phi_{\mathcal C}$ (and formally declaring the other intersection numbers to be zero). We sometimes use this more general notion.
\end{remark}

For each $1 \leq i \leq n-1$, we define 
\[
	S_i = \{(A_1, \dots, A_{n-1}) \in \mathcal{CQ}(V) \colon \mathrm{rk}(A_i) = 1\}
    = \overline{\{(A_1, \dots, A_{n-1}) \in \mathcal{CQ}(V) \colon \mathrm{rk}(A_1) = i\}}.
\]
The subvariety $S_i$ coincides with the strict transform of the exceptional locus of the blow-up $X_{i-1} \leftarrow X_i$. The subvarieties $S_i$ define divisors on $\mathcal{CQ}(V)$ which form a basis of $\Pic_{\mathbb{Q}}(\mathcal{CQ}(V))$. To simplify the notation, we denote their classes in $\Pic_{\mathbb{Q}}(\mathcal{CQ}(V))$ by $S_i$ as well. 

Considering $\mathcal{CQ}(V)$ as a Chow quotient of $\mathcal{LG}(V)$, the divisor $S_i$ can be interpreted as chains of orbits
whose union intersects the $i$-th embedded Grassmaninan $\mathrm{Gr}(i, V)$. We sometimes say that those are the orbits \textit{breaking at the $i$-th Grassmannian}.

The basis $\{S_i\}$ is related to the basis $\{L_i\}$ by the following relations, that were already known to Schubert \cite{Schubert} (see \cite[Theorem 3.13]{Massarenti} for the modern approach):
\[
	S_i = -L_{i-1} + 2 L_i - L_{i+1},
\]
where we set $L_0 = L_n = 0$. \\ %TODO add reference for this relation [https://arxiv.org/abs/1803.09161]
% TODO we first did intersection L_i's on CC(V), then we define S_i on CQ(V) and have relations involving both, which is not good. Later we do everything in CQ so I don't know how to better handle it but we need to make it clearer 

Let us recall the notion of a tensor and describe the general construction which identifies algebras with certain tensors and tensors with linear subspaces of matrices. 
\begin{definition}
    A $3$-order tensor is an element $T \in U \otimes V \otimes W$ of the tensor product of finitely dimensional $\CC$-vector spaces $U, V, W$.
\end{definition} 

Each tensor $T$ gives rise to three linear maps: 
\[
    T_U \colon U^* \to V \otimes W,\quad T_V \colon V^* \to U \otimes W,\quad T_W \colon W^* \to U \otimes V.
\]
The images of these maps are called the \emph{flattenings} of $T$, with respect to $U, V, W$, respectively. A tensor can be recovered from its flattening up to an isomorphism. A tensor $T \in U \otimes V \otimes W$ is called \emph{$1_U$-generic} if $\dim_\CC V = \dim_\CC W$ and the flattening of $T$ with respect to $U$ contains an invertible linear map $T_U(u) \colon V^* \to W$ (and similarly for the other factors). A tensor is called \emph{$1$-generic} if $\dim U = \dim V = \dim W$ and it is $1_U$, $1_V$ and $1_W$-generic. 

If $U = V$ are $n$-dimensional and the flattening with respect to $W$ contains a matrix of rank at least $n-2$, then we can consider its characteristic numbers.
\begin{definition}
    The characteristic numbers of the tensor $T \in V \otimes V \otimes W$ are the characteristic numbers of the flattening of $T$ with respect to $W$.
\end{definition}

Let $A$ be an algebra of dimension $n$. The multiplication $A \times A \ni (a, b) \mapsto ab \in A$ is a bilinear map, so it corresponds to a tensor.
\begin{definition}
    The structure tensor of the algebra $A$ is the tensor $T^A \in A^* \otimes A^* \otimes A$ corresponding to the multiplication in $A$.
\end{definition}
We assume that $A$ is commutative, so the flattenings of $T^A$ with respect to the two first factors coincide and the flattening corresponding to the third factor lies in the subspace $S^2 A^* \subset A^* \otimes A^*$. We denote these flattenings by $L^{A^*}$ and $L^{A}$. We sometimes call $L^A$ the \emph{multiplication table} of $A$. We denote the variety $X_{L^A}$ by $X_A$.
\begin{example}
    Consider the algebra 
    \[
        A = \CC[x,y,z]/(x^2, y^2, xz, yz, z^2 - xy).
    \]
    Choose the basis $a_0 = 1, a_1 = x, a_2 = y, a_3 = z, a_4 = z^2$ of the vector space $A$. The structure tensor of the algebra $A$ is 
    \[
        T^A 
        = a_0^* \otimes a_0^* \otimes a_0 
        + \left( \sum_{i = 1}^4 (a_0^* \otimes a_i^* + a_i^* \otimes a_0^*) \otimes a_i \right) 
        + (a_1^* \otimes a_2^* + a_2^* \otimes a_1^* + a_3^* \otimes a_3^*) \otimes a_4.
    \]
    The first two factors correspond to the trivial relations $1 \cdot a = a \cdot 1 = a$ and the last one correspond to $x \cdot y = y \cdot x = z \cdot z$ (all of which are equal to $xy = z^2$). The multiplication table of $A$ can be represented as a matrix
    \[
        \begin{bmatrix}
            a_0 & a_1 & a_2 & a_3 & a_4 \\
            a_1 & 0 & a_4 & 0 & 0 \\
            a_2 & a_4 & 0 & 0 & 0 \\
            a_3 & 0 & 0 & a_4 & 0 \\
            a_4 & 0 & 0 & 0 & 0 \\
        \end{bmatrix}.
    \]
\end{example}

\begin{definition}
    The characteristic numbers of algebra $A$ are the characteristic numbers of its structure tensor $T^A$. The characteristic sequence of $A$ is the characteristic sequence of $T^A$. 
\end{definition}

% \begin{definition}
%     Let $T \in \otimes V \otimes V \otimes W$ be such a tensor that the contraction $W^*(T) \subseteq V \otimes V$ contains a matrix of rank at least $n-2$, where $n = \dim V$. Let $d = \dim W^*(T)$. We define the polynomial on $\Pic_{\mathbb{Q}}(\mathcal{C}\mathcal{C}(V))$: $$P_T(D) := SD^d \in H^0(\mathcal{C}\mathcal{C}(V), \mathbb{Q}) \simeq \mathbb{Q},$$ where $S$ is the Poincar\'e dual of the strict transform of $W^*(T)$ by $\pi_1$.

%     In the basis $L_i$ we have $$P_T\left(\sum_{i=1}^{n-1}a_iL_i\right) := \sum_{b_1 + \ldots + b_{n-1} = d}\binom{d}{b_1,\ldots,b_{n-1}}S\prod_{i=1}^{n-1}L_i^{b_i}.$$
%     We call the coefficient $c^T_{b_1,\ldots,b_{n-1}} := S\prod_{i=1}^{n-1}L_i^{b_i}$ the \textit{characteristic numbers} of the tensor $T$ and write simply $c_{b_1,\ldots,b_{n-1}}$ when it is clear what tensor $T$ we consider.

%     Let $A$ be an n-dimensional algebra and $T_A$ be its structure tensor. We call the coefficients $c^{T_A}_{b_1,\ldots,b_{n-1}}$ \textit{characteristic numbers} of the algebra $A$.
% \end{definition}

\section{Characterisation of algebras}

In this section, we tackle the natural problem of extracting information about properties of algebras from their characteristic numbers. We characterise Gorenstein and local complete intersection algebras in terms of their characteristic sequence.

\subsection{Gorenstein algebras}

% \begin{lemma}
%     Put here known things from paper: VANISHING HESSIAN, WILD FORMS AND THEIR BORDER VSP
% \end{lemma}
 
%Let us grade $A$ by the powers of $m$ with $\soc(A) = m^d$.
%Being Gorenstein is equivalent to all multiplication maps $$A_t \times A_{d-t} \to  \soc(A) \cong \mathbb{C}$$ being perfect pairings. This is also equivalent to the fact that there an $A$-module isomorphism $A \cong A^*$, where the structure of an $A$-module on $A^*$ is given by $a \cdot l(\cdot) := l(a\cdot)$. For a Gorenstein algebra this isomorphism is given by $A_t \ni a \mapsto f_t(a) \in A^*_{d-t}$, where $f_t: A_t \to A_{d-t}^*$ comes from the perfect pairing above. 

%It is also convenient to understand how the matrix of multiplication looks like for a Gorenstein algebra $A$. Using the isomorphisms $A_t \cong A_{d-t}$ and backward induction from $d$ to $0$ we see that there is a unique non-zero summand in the $n \times n$ minor of $T_A$ which is equal to $z^n$ with $(z) = \soc(A)$ - for more details see the proof of Lemma \ref{onepoint}.

Assume that $A$ is a Gorenstein algebra and choose a $\lambda \in A^*$ inducing a perfect pairing $A \times A \to \CC$. This pairing gives a linear isomorphism $G \colon A \to A^*$, which yields an isomorphism $A^* \otimes A^* \otimes A \cong A^* \otimes A^* \otimes A^*$. In the following section we often use this identification to consider whether $T^A$ is symmetric or to talk about characteristic numbers of the subspace $L^{A^*}$.\\

We start by a characterisation in terms of structure tensors. Some of these implications are proved in \cite{VSP}, but we reprove it here for the sake of completeness.
\begin{lemma}\label{l:gorenstein_tensor}
    Let $A$ be an algebra. The following conditions are equivalent:
    \begin{enumerate}
        \item algebra $A$ is Gorenstein,
        \item\label{l:gorenstein_tensor:1} the flattening $L^A$ of the structure tensor $T^A$ contains an invertible matrix,
        \item\label{l:gorenstein_tensor:2} the structure tensor $T^A$ is 1-generic,
        \item\label{l:gorenstein_tensor:3} the structure tensor $T^A$ is symmetric.
    \end{enumerate}
\end{lemma}
To be precise, condition \ref{l:gorenstein_tensor:3} means that there exists an isomorphism $A^* \cong A$ such that the structure tensor $T^A$, consider as a tensor in $A^* \otimes A^* \otimes A^*$, is symmetric.
\begin{proof}
    Recall that there are two maps 
    \[
        (T^A)_A \colon A^* \to S^2 A^* \quad \text{and}\quad (T^A)_{A^*} \colon A \to A^* \otimes A
    \]
    corresponding to the structure tensor $T^A$. 
        
         Condition \ref{l:gorenstein_tensor:1} is equivalent to $A$ being a Gorenstein algebra because the map $(T^A)_A$ gives a bijective correspondence between functionals $\lambda \in A^*$ inducing the perfect pairing and invertible matrices in $L^A$.
         The structure tensor $T^A$ is always $1_{A^*}$ generic, because $(T^A)_{A^*}$ maps the identity of $A$ to the identity map in $\operatorname{Hom}(A, A) \cong A^* \otimes A$. Therefore condition \ref{l:gorenstein_tensor:2} is equivalent to $T^A$ being $1_A$-generic, which is precisely condition \ref{l:gorenstein_tensor:1}.
        
         Assume that condition \ref{l:gorenstein_tensor:1} holds. Pick $\lambda \in A^*$ such that 
         $(T^A)_A(\lambda)$ is of full rank. As noted above, this is equivalent to saying that 
        \[
            G \colon A \ni a \mapsto (b \mapsto \lambda(ab)) \in A^*
        \]
        is an isomorphism. After applying this isomorphism to the last factor of $A^* \otimes A^* \otimes A$, the structure tensor $T^A$ can be viewed as a trilinear map
        \[
        T \colon A \times A \times A \ni (a, b, c) \mapsto G(c)(ab) \in \mathbb{C}.
        \]
        By definition of $G$ we have
        \[
            T(a, b, c) = G(c)(ab) = \lambda(abc),
        \]
        which does not depend on the permutation of $a, b, c$ as $A$ is commutative. Hence condition \ref{l:gorenstein_tensor:1} implies condition \ref{l:gorenstein_tensor:3}.
         We already proved that $T^A$ is always $1_{A^*}$-generic, so if condition \ref{l:gorenstein_tensor:3} holds, then $T^A$ is          $1$-generic. Thus condition \ref{l:gorenstein_tensor:3} implies condition \ref{l:gorenstein_tensor:2}.
\end{proof}

More surprisingly, Gorenstein algebras can be also characterised by their characteristic numbers. Part of it can be also deduced from \cite{Jordan}, using the theory of Jordan algebras. We start with an auxiliary lemma and then give the characterisation. % TODO Should we write more precisely which parts are from Jordan algebras?

\begin{lemma}\label{l:symmetric}
    Let $A$ be an algebra. The characteristic sequence for $L^{A^*}$ is symmetric.
\end{lemma}
\begin{proof}
    An invertible element $M \in L^{A^{*}}$ corresponds to a matrix of multiplication by a unit element in $A$, so we have $M^{-1} \in L^{A^*}$ as well. Matrices of rank at most $n-1$ form a codimension $1$ subspace of $L^{A^*}$, hence they do not contribute to the multidegree and we can consider the gradient restricted to the space of invertible matrices in $L^{A^*}$. There, the projectivization of the gradient map corresponds to matrix inversion, so the image of this map is symmetric and hence the sequence for $L^{A^*}$ is symmetric. 
\end{proof}

\begin{proposition}\label{p:gorenstein}
    Let $A$ be an algebra. The following conditions are equivalent:
    \begin{enumerate}
        \item\label{l:gorenstein_numbers:0} The algebra $A$ is Gorenstein,
        \item\label{l:gorenstein_numbers:1} all the characteristic numbers of $A$ and $L^{A^*}$ coincide,
        \item\label{l:gorenstein_numbers:2} the characteristic sequences of $A$ and $L^{A^*}$ coincide,
        \item\label{l:gorenstein_numbers:3} the characteristic sequence of $A$ is symmetric,
        \item\label{l:gorenstein_numbers:4} the characteristic sequence of $A$ ends with 1,
        \item\label{l:gorenstein_numbers:5} the characteristic sequence of $A$ ends with a non-zero number.
    \end{enumerate}
\end{proposition}
\begin{proof}
    Consecutive implications from $A$ being Gorenstein to the condition \ref{l:gorenstein_numbers:5} are straightforward.
    Indeed, if $A$ is Gorenstein, then Lemma \ref{l:gorenstein_tensor} implies that $T$ is symmetric, so the condition \ref{l:gorenstein_numbers:1} holds.
    Condition \ref{l:gorenstein_numbers:1} trivially implies condition \ref{l:gorenstein_numbers:2}.
    Condition \ref{l:gorenstein_numbers:2} combined with Lemma \ref{l:symmetric} immediately imply condition \ref{l:gorenstein_numbers:3}.
    Condition \ref{l:gorenstein_numbers:3} implies condition \ref{l:gorenstein_numbers:4} because the first term of a characteristic sequence of a linear subspace is always $1$, which follows from interpreting the first characteristic number as the degree of the identity map $\PP(S^2 A^*) \to \PP(S^2 A^*)$.
    Condition \ref{l:gorenstein_numbers:4} trivially implies condition \ref{l:gorenstein_numbers:5}.

    The implication from condition \ref{l:gorenstein_numbers:5} to $A$ being Gorenstein is more involved. Let $n := \dim_{\mathbb{C}}(A)$ and $A = \bigoplus_{i=1}^k A_i$, where each $A_i$ is a local algebra with the maximal ideal $\maxideal_i$. Assume that condition \ref{l:gorenstein_numbers:5} holds and $A_1$ is not Gorenstein. Let $c_1 := \dim_{\CC} A_1$.  
    Choose a basis $\{a_j\}_{j=1,\dots ,c_1}$ of the algebra $A_1$ such that consecutive subsequences of $a_j$ form a basis of
    \[ A_1/\maxideal_1, \maxideal_1/\maxideal_1^2,\dots , \maxideal_1^{d_1-1}/\maxideal_1^{d_1}, \maxideal_1^{d_1},\]
    where $\maxideal^{d_1} \neq (0)$ and $\maxideal^{d_1+1} = (0)$. We choose any basis of the remaining algebras $A_2, \ldots, A_{k}$ and consider elements of $L^A$ as matrices in this basis. 
    %Note that $M$ is a block-diagonal matrix with blocks $M_i$, multiplication tables of algebras $A_i$.
    Let $N_{A} \subseteq L^A$ be the subset of matrices of rank $n-1$. Since for any rank $n-1$ square matrix $M$ of size $n$ we have $\ker(M) = \im(\adj(M))$,
    there exist $\lambda_1, \ldots, \lambda_n \in \mathbb{C}$ such that
    \[ (\grad \det)(M) = \adj(M) = [\lambda_1 v_M \: | \: \lambda_2 v_M \: | \ldots | \lambda_n v_M] \]
    with $\langle v_M \rangle = \ker M$ for any $M \in N_A$. Since any $M \in N_A$ is a block matrix with blocks corresponding to local algebras $A_i$ and $\dim_{\CC} \soc(A_1)>1$, the only non-zero entries in the rows $c_1$ and $c_1-1$ of $M$ are in the first column of $M$. This implies that for any such $M$ the first coordinate of $v_M$ is zero (as the minor of $M$ corresponding to the first column and $c_1$-th row is zero), and it follows that
    \[ \dim \mathbb{P}(\grad \det)(N_A) \leq n-2. \]
    Therefore, the matrices of rank at most $n-1$ do not contribute to the degree $c_{0,0,\ldots, n-1}$. Since the degree is non-zero, there must be an invertible element in $L^A$. Therefore, Lemma \ref{l:gorenstein_tensor} implies that $A$ is Gorenstein. 
\end{proof}

As noted in Example \ref{e:grad}, the last characteristic number of an algebra $A$ is the degree of the restriction of the gradient $(\grad \det)|_{L^A}$. Proposition \ref{p:gorenstein} says that for Gorenstein algebras, this number is equal to $1$. One could also consider the gradient of the restriction $\grad (\det|_{L^A})$. It turns out that this map is particularly simple.

\begin{lemma}\label{l:onepoint}
    Let $A$ be a local Gorenstein algebra. Then the gradient of the restriction $\grad (\det|_{L^A})$ is a constant map.
\end{lemma}

\begin{proof}
    
    Let $d$ be the greatest number such that $\maxideal^d \neq 0$ (recall that by our assumption $\soc(A) = \maxideal^d$ is one-dimensional). Choose a basis $a_1, \dots, a_n$ of $A$ in such a way that consecutive subsequences of $(a_1, \dots, a_n)$ are bases of the vector spaces $A/ \maxideal, \maxideal/\maxideal^2, \dots, \maxideal^d/\maxideal^{d+1} = \maxideal^d$. In particular, $a_1$ corresponds to an invertible element in $A$ and $a_n$ corresponds to a generator of the socle.
    
    Let $M$ be the multiplication table, thought of as a matrix in variables $a_i$. We claim that $\det M$ is a non-zero multiple of $a_n^n$. Let us recall a property of (not necessarily homogeneous) Gorenstein algebras that follows from \cite[Chapter 2]{Iarrobino}: The perfect pairing \(\mu: A \times A \to \mathbb{C}\) of a Gorenstein algebra $A$ induces the perfect pairing $$\bar{\mu}: \maxideal^k/\maxideal^{k+1} \times (0:\maxideal^{k+1})/(0: \maxideal^k) \to \mathbb{C}$$ $$(x_1 + \maxideal^{k+1}, x_2 + (0: \maxideal^k)) \mapsto \mu(x_1, x_2)$$ for any $k \leq d$.

    Now, let $y_1, \ldots, y_n$ be a basis of $A$ such that the consecutive subsequences of $(y_1, \ldots, y_n)$ are bases of the vector spaces \(A/(0:\maxideal^d), (0:\maxideal^d)/(0:\maxideal^{d-1}), \ldots, (0:\maxideal) = \maxideal^d\) and $N: A \to A$ be the linear transformation such that $N(a_j) = y_j$ for $j = 1 ,\ldots, n$. Let $M' = M \cdot N^{\top}$, where we regard $N$ as the matrix of the map $N$ in the basis $a_1,\ldots,a_n$ of $A$. Since $N$ is invertible it is enough to show the claim for $M'$.
    
    Recall that the entry $m'_{k,l}$ of the matrix $M'$ is the result of the multiplication of $a_k$ and $y_l$. We associate a pair of weights to each entry $(k,l)$ of $M'$ by the formula $$(w_{kl}^1, w_{kl}^2) = (\max_s\{0 \leq s \leq d \mid a_k \in \maxideal^s\}, \min_s\{0 \leq s \leq d \mid y_l \in (0:\maxideal^{s+1})\}).$$
    Let $F$ be a monomial in $\det M'$ corresponding to the choice of indices $(i_1, j_1), \ldots ,(i_n, j_n)$. It is enough to show that $F$ is a (possibly zero) multiple of $a_n^n$. Note that it follows from the induced perfect pairing above that $\dim_{\mathbb{C}}(\maxideal^s/\maxideal^{s+1}) = \dim_{\mathbb{C}}((0:\maxideal^{s+1})/(0:\maxideal^s))$ for each $s$. Therefore, we have $$\sum_{k=1}^n w_{i_k, j_k}^1 = \sum_{k=1}^n w_{i_k, j_k}^2.$$ We claim that if $F \neq 0$ then $w_{i_k, j_k}^1 = w_{i_k, j_k}^2$ for all $k$. If this is not the case, then there exists $k$ such that $w_{i_k, j_k}^1 > w_{i_k, j_k}^2$. Then $m'_{i_k, j_k} = 0$, since this entry is the result of the multiplication of an element in $(0:\maxideal^{w_{i_k, j_k}^2 + 1})$ and an element in $\maxideal^{w_{i_k, j_k}^1} \subseteq \maxideal^{w_{i_k, j_k}^2 + 1}$. It follows that for each $k$ we have $m'_{i_k, j_k} \in \maxideal^{w^1_{i_k, j_k}} \cdot (0:\maxideal^{w^1_{i_k, j_k} + 1}) \subseteq m^d = \langle a_n \rangle$, hence $F = c \cdot a_n^n$ for some $c \in \mathbb{C}$, which proves the claim for $\det M$.
    
    Now, since $a_1, \ldots, a_n$ can be regarded as the coordinates of $L^A$, it follows that
    \[
        \grad (\det|_{L^A}) \colon [a_1:\ldots:a_n] \mapsto [0:\ldots:0: a_n^{n-1}],
    \]
    hence the image is just one point.
\end{proof}

\begin{remark}
    Note that we can write every finite algebra as a direct sum $A \simeq A_1 \oplus \dots \oplus A_r$ and then $L^{A_{\maxideal_1} \oplus \dots \oplus A_{\maxideal_r}} \cong L^{A_{\maxideal_1}} \oplus \dots \oplus L^{A_{\maxideal_r}}$. Applying the results above to each of the summands of $A$, we see that the determinant of the multiplication table is a monomial of the form $a_1^{d_1} \cdot \dots \cdot a_r^{d_r}$, where $d_i := dim_{\CC}(A_i)$. Thus the gradient map is given by a sum $\sum_{j=1}^r a_1^{d_1}\cdot \hdots a_j^{d_j-1} \hdots \cdot a_r^{d_r}$. Thus the characteristic sequence for $A$ depends only on the dimensions of the indecomposable summands $A_i$. If $A$ were non-Gorenstein, then $L_A$ would not contain an invertible matrix, and thus the determinant restricted to $L_A$ would be zero, in particular its gradient wouldn't be well-defined.
\end{remark}

\subsection{Local complete intersections}
First, we give two auxiliary lemmas.
\begin{lemma}\label{l: 2 x 2 minors}
    Let $A$ be an algebra and let $I \subset \Sym A^*$ be an ideal generated by some linear 
    combinations of $(2 \times 2)$-minors of the multiplication table of $A$. Then there exists a 
    surjective map $(\Sym A^*)/I \to A$, which is an isomorphism if and only if the ideal $I$ coincides with the ideal generated by all $(2 \times 2)$-minors.
\end{lemma}
\begin{proof}
    Choose a basis $a_0, \dots, a_{n-1}$ of $A$ such that $a_0$ corresponds to $1 \in A$, let 
    $x_0, \dots, x_{n-1}$ be the dual basis and identify $\Sym A^*$ with 
    $S = \CC[x_0, \dots, x_{n-1}]$.
    Consider the surjective map $S \to A$ defined on generators by $x_i \mapsto a_i$. Let 
    $\ell_{kl} \in S_1$ be the entry of the multiplication table corresponding to $a_k a_l$. 
    Each $2 \times 2$ minor is indexed by two pairs $i_1 < j_1, i_2 < j_2$. 
    The minor $\ell_{i_1 i_2} \ell_{j_1 j_2} - \ell_{i_1 j_2} \ell_{i_2 j_1}$ is mapped to 
    $a_{i_1}a_{i_2}\cdot a_{j_1}a_{j_2} - a_{i_1}a_{j_2} \cdot a_{i_2} a_{j_1} = 0$, so the map 
    $S \to A$ factors through a map $S / I \to A$ which is still surjective. 

    The multiplication table specifies the algebra, so all relations in $A$ are generated by relations of the form $a_k a_l = \sum_i \lambda_i^{kl} a_i = a_0 \sum_i \lambda_i^{kl} a_i$, which are non-trivial only for $k, l > 0$. Such relations are encoded by the minors indexed by pairs $0 < k, 0 < l$, that is, $x_k x_l = \ell_{0k} \ell_{0l} = \ell_{00} \ell_{kl} = x_0 \sum_i \lambda_i^{kl} x_i$. Therefore if $I$ is generated by all $(2 \times 2)$-minors, then the inverse map $A \to S/I$ is well-defined, so we have an isomorphism $S/I \cong A$. If $I$ is strictly contained in the ideal generated by all $(2 \times 2)$-minors, then $\dim_\CC S / I > \dim_\CC A$, so there is no isomorphism.
\end{proof}

\begin{lemma}\label{l:degree}
    Let $A$ be a local algebra and consider the map $\varphi \colon \PP (L^A) \dashrightarrow \PP (S^2 (\bigwedge^2 A^*))$. If $A$ has non-trivial multiplication (that is, $\maxideal^2 \neq 0$), then the map $\varphi$ has degree $1$. 
\end{lemma}
\begin{proof}
    Let $\maxideal \subset A$ be the maximal ideal and let $d$ be the largest integer such that 
    $\maxideal^d \neq 0$. Choose a basis $a_0, \dots, a_{n-1}$ of $A$ so that $a_0 = 1$, 
    $a_j \in \maxideal$ for $j > 0$ and $a_{n-1}$ is an element from $\maxideal^d$.
    Let $x_0, \dots, x_{n-1}$ be the dual basis. 
    We have $a_j a_{n-1} = 0$ for $j > 0$, so the $2 \times 2$-minor of the multiplication table 
    indexed by $0, n-1$ and $0, j$ is $x_j x_{n-1}$. 
    
    Consider $\varphi$ restricted to $\{x_0 \neq 0, x_{n-1} \neq 0\}$. If $s = [1 : s_1 : \dots :s_{n-1}], t = [1 : t_1 : \dots : t_{n-1}]$ are mapped to the same point, then in particular $[s_1 s_{n-1} : \dots : s_{n-1}^2] = [t_1 t_{n-1} : \dots : t_{n-1}^2]$, so there exists a non-zero scalar $\lambda$ such that $(s_1, \dots, s_{n-1}) = \lambda (t_1, \dots, t_{n-1})$. We assumed that $A$ has non-trivial multiplication, then there is also a minor of the form $x_0 l - x_i x_j$ for some non-zero linear form $l$. We have
    \[
        \varphi(t) = [t_1 t_{n-1} : \dots : t_{n-1}^2 : l(t) - t_i t_j : \dots] = [\lambda^2 t_1 t_{n-1} : \dots : \lambda^2 t_{n-1}^2 : \lambda^2(l(t) - t_i t_j) : \dots]
    \]
    and 
    \[
        \varphi(s) = [\lambda t_1  \cdot \lambda t_{n-1} : \dots : \lambda t_{n-1} \cdot \lambda t_{n-1} : l(\lambda t) - \lambda t_i \cdot \lambda t_j : \dots],
    \]
    so
    $
        \lambda l(t) - \lambda^2 t_i t_j = \lambda^2(l(t) - t_it_j),
    $
    which implies that $\lambda = 1$. It means that the map $\varphi$ is generically injective and hence it has degree $1$. 
\end{proof}

The main characterisation is the following. 

\begin{proposition}\label{p:ci}
    Let $A$ be an $n$-dimensional local algebra. Then 
    \[
    		c^A_{0,n-1,0,\dots,0} \leq 2^{n-1} - n
    \] 
    and the equality holds if and only if $A$ is a (local) complete intersection.
\end{proposition}
\begin{proof} 
    First assume that $A$ has non-trivial multiplication.
    Let $B$ denote the base locus of $(2 \times 2)$-minors of the multiplication table of $A$ and let $e(B)$ be the algebraic multiplicity of $B$. As noted in the proof of Lemma \ref{l:degree}, the set of $(2 \times 2)$-minors of the multiplication table contains equations $x_1 x_{n-1}, x_2x_{n-1}, \dots, x_{n-1}^2$, the hypersurfaces defined by these equations are transversal, so by B{\'e}zout's theorem, a system of $n-1$ general linear combinations of $(2 \times 2)$-minors of $T^A$ has $2^{n-1}$ solutions. The number of points outside of $B$ is exactly $[X_A] \cdot L_2^{n-1}$, because the image in $\PP (S^2 (\bigwedge^2 A^*))$ is parameterized by $(2 \times 2)$-minors of the multiplication table of $A$ which do not vanish simultaneously. Therefore by Lemma \ref{l:degree}
    \[
        [X_A] \cdot L_2^{n-1} 
        = \frac{2^{n-1} - e(B)}{\deg \varphi}
        = 2^{n-1} - e(B)
        \leq 2^{n-1} - \deg B
    \]
    and the equality holds if and only if the ideal $I$ of $B$ is a complete intersection, see \cite[Section 6]{gesmundo2023collineationvarietiestensors}.
    By Lemma \ref{l: 2 x 2 minors} we get that $I$ satisfies $\Sym A^* / I \cong A$, so $\deg B = \dim_\CC A = n$, which gives the inequality $[X_A] \cdot L_2^{n-1} \leq 2^{n-1} - n$. Being a complete intersection is independent of the embedding, so the equality holds if and only if $A$ is a complete intersection.

    If $A = \CC[x_1, \dots, x_{n-1}] / (x_1, \dots, x_{n-1})^2$, then by Example \ref{e:trivial} we have $[X_A] \cdot L_2^{n-1} = 0$, which is consistent with the fact that $A$ is not a complete intersection. 
\end{proof}

\section{Calculation of characteristic numbers}

\subsection{Recursive formula for characteristic sequences}
In this section we derive a formula expressing characteristic sequences of a direct sum of subspaces of symmetric matrices in terms of characteristic sequences of direct summands, under a technical assumption is that both subspaces contain invertible matrices. This applies in particular to subspaces coming from multiplication in Gorenstein algebras (see Lemma~\ref{l:gorenstein_tensor}),
so it reduces the problem of calculation characteristic sequences of Gorenstein algebras to calculating them for
the local ones.

\begin{definition}
    Let $V$ be a finite dimensional vector space. For any subspace $L \subset S^2 V$ of dimension $n+1$ containing an invertible matrix, let $Y_L$ be the graph of the rational map
    \[
        \grad \det \colon \PP(S^2 V) \dashrightarrow \PP (S^2V^*) \cong \mathbb{P}\left(S^2\left(\bigwedge^{n-1}V\right)\right),
    \]
    restricted to $\mathbb{P}(L)$. On the open subset of invertible matrices $U_L \subset \mathbb{P}(L)$ the map $\grad \det$ sends $[A] \mapsto [A^{-1}]$, as in Example \ref{e:grad}.
\end{definition}
Alternatively, $Y_L$ can be described as the image 
of $X_L \subset \mathcal{CQ}(V)$ under projecting to $\PP(S^2 V) \times \PP(S^2V^*)$. We denote the characteristic sequence of $L$ by
\[ \mu^L_k = c^L_{n-k,0,\dots,k} \text{ for } k=0,\dots,n . \]
They are equal to the coefficients of $[\PP^{n-k} \times \PP^k]$ in the expansion in the Chow ring of the class
\[[Y_L] \in \CH(\PP(S^2 V) \times \PP(S^2V^*)). \]
We sometimes omit $L$ from the notation if it is clear from the context.\\

Suppose we have two such $V_i$ and $L_i \subset S^2 V_i$.
%We have embeddings 
%\[
%    \iota_i \colon\PP(S^2V_i) \times \PP(S^2V_i^*) \hookrightarrow \PP(S^2 V_1 \oplus S^2V_2) \times \PP(S^2 V_1^* \oplus S^2V_2^*).
%\]
Our aim is to describe the class in the Chow ring
\[
    [Y_{L_1 \oplus L_2}] \in \CH(\PP(S^2 V_1 \oplus S^2V_2) \times \PP(S^2 V_1^* \oplus S^2V_2^*))
\]
in terms of the classes
\[
    [Y_{L_i}] \in \CH(\PP(S^2 V_i) \times \PP(S^2V_i^*)).
\]
Let us introduce a notion that, in our setting, serves as a right analogue of the join of two subvarieties in a projective space.
\begin{definition}
    For $i=1,2$ let $X_i \subset \PP (S^2V_i) \times \PP (S^2 V_i^*)$ be subvarieties. We define their join
    \[J(X_1, X_2) \subset \PP(S^2 V_1 \oplus S^2V_2) \times \PP(S^2 V_1^* \oplus S^2V_2^*)\]
    to be the closure of the locus of points $([v_1:v_2],[w_1:w_2])$ such that 
    $(v_1,w_1) \in X_1$ and $(v_2, w_2)\in X_2$.
\end{definition}
%This definition slightly differs from the standard terminology when viewing $\mathbb{P}(S^2V_i)$ as a subspace of $\mathbb{P}(S^2V_1 \oplus S^2V_2)$. This is because the varieties we are considering lie in a product of two projective spaces.

\begin{lemma}\label{invariant_under_rational_equivalence}
    Let $V_i, X_i$ be as above and let $X_i' \subset \PP (S^2V_i) \times \PP (S^2 V_i^*)$ be any closed subvarieties such that $[X_i] = [X_i']$ in the Chow ring $\CH(\PP (S^2V_i) \times \PP (S^2 V_i^*))$. Then $[J(X_1, X_2)] = [J(X_1', X_2')]$ in the Chow ring $\CH(\PP(S^2 V_1 \oplus S^2V_2) \times \PP(S^2 V_1^* \oplus S^2V_2^*))$.
\end{lemma}
\begin{proof}
    % Let $\{X_i\}_i, \{X'_{i'}\}_{i'}$ and $ \{Y_j\}_j$ be the irreducible components of $X, X'$ and $ Y$, respectively. Let $U \subset (\PP V \times \PP V') \times \PP^1$ be a variety giving a rational equivalence between $X$ and $X'$. Let $U_j \subset (\PP (V \oplus W) \times \PP(V' \oplus W')) \times \PP^1$ be the join of images of $U$ and $Y_j$. Both $U$ and $Y_j$ are irreducible, so $U_j$ is irreducible and gives a rational equivalence between $J(X, Y_j)$ and $J(X', Y_j)$. Note that $\{J(X_i, Y_j)\}_{i,j}$ are the irreducible components of $J(X, Y)$ and $\{J(X_i, Y_j)\}_i$ are the irreducible components of $J(X, Y_j)$ (and similarly for $X'$), so 
    % \[
    %     [J(X, Y)] = \sum_{i, j} [J(X_i, Y_j)] = \sum_{j} [J(X, Y_j)] = \sum_{j} [J(X', Y_j)] = \sum_{i', j} [J(X'_{i'}, Y_j)] = [J(X', Y)].
    % \]
    Let $ \{(X_2)_j\}_{j\in J}$ be the set of the irreducible components of $X_2$. Let
    \[ U \subset \PP (S^2V_1) \times \PP (S^2 V_1^*) \times \PP^1 \]
    be any closed subvariety witnessing the rational equivalence between $X_1$ and $X_1'$. Consider the projections
    \[ \pi_1: \PP (S^2V_1) \times \PP (S^2 V_1^*) \times \PP^1 \rightarrow \PP (S^2V_1) \times \PP (S^2 V_1^*) \text{ and } \pi_2:  \PP (S^2V_1) \times \PP (S^2 V_1^*) \times \PP^1 \rightarrow  \PP^1.\]
    For each $j\in J$, let
    \[ U_j \subset \PP(S^2 V_1 \oplus S^2V_2) \times \PP(S^2 V_1^* \oplus S^2V_2^*) \times \PP^1 \]
    be the subvariety defined as
    \[\bigcup_{s\in \PP^1} J\left(\pi_1(\pi_{2}^{-1}(s)\cap U),(X_2)_j\right) . \] 
    Both $U$ and $(X_2)_j$ are irreducible, so $U_j$ is irreducible and gives a rational equivalence between $J(X_1, (X_2)_j)$ and $J(X_1', (X_2)_j)$. Hence
    \[
        [J(X_1, X_2)]  = \sum_{j} [J(X_1, (X_2)_j)] = \sum_{j} [J(X_1', (X_2)_j)]  = [J(X_1', X_2)].
    \]
    By applying the same argument to $X_2$ and $X_2'$ we conclude that
    \[ 
        [J(X_1, X_2)] = [J(X_1', X_2)] = [J(X_1', X_2')]. \qedhere
    \]
\end{proof}
For $i=1,2$ let
\[ N_i := \dim \PP (S^2 V_i) = \dim \PP(S^2 V_i^*) \text{ and } N := \dim \PP(S^2 V_1 \oplus S^2V_2) = \dim \PP(S^2 V_1^* \oplus S^2V_2^*), \]
so $N = N_1 + N_2 +1$. Let
\[ n_i := \dim Y_{L_i} = \dim L_i -1 \text{ and } n := \dim Y_{L_1 \oplus L_2} = \dim (L_1 \oplus L_2) - 1, \]
so $n = n_1 + n_2 + 1$.
The varieties $Y_{L_i}$ are irreducible (see \ref{l:irreducible}) and have codimension $2N_i - n_i$.
Recall that there are isomorphisms
\[
    \CH(\PP (S^2V_i) \times \PP (S^2 V_i^*)) 
    = \ZZ[x_i, y_i]/(x_i^{N_i+1}, y_i^{N_i+1})
\]
and
\[
    \CH(\PP (S^2 V_1 \oplus S^2V_2) \times \PP (S^2 V_1^* \oplus S^2V_2^*)) 
    = \ZZ[x, y]/(x^{N+1}, y^{N+1}),
\]
where $x_i, y_i, x, y$ correspond to the classes of hyperplanes on the respective factors. Under this isomorphisms
\[
    \CH(\PP (S^2V_i) \times \PP (S^2 V_i^*)) 
    \ni [\PP^{k}\times \PP^l] \mapsto x_i^{N_i-k}y_i^{N_i-l} \in  \ZZ[x_i, y_i]/(x_i^{N_i+1}, y_i^{N_i+1})
\]
and
\[
    \CH(\PP (S^2 V_1 \oplus S^2V_2) \times \PP (S^2 V_1^* \oplus S^2V_2^*))\ni [\PP^{k}\times \PP^l] \mapsto x^{N-k}y^{N-l} \in \ZZ[x, y]/(x^{N+1}, y^{N+1}).
\]

The following lemma confirms that the join construction introduced above is a natural generalization of the classical case, in which the ambient space is a single projective space rather than a product of two.

\begin{lemma}\label{l:join_formula}
    For $i=1,2$ let
    \[
        [Y_{L_i}] = \sum_{k=0}^{n_i} \mu_k^i[\mathbb{P}^{n_i-k} \times \mathbb{P}^k]=\sum_{k=0}^{n_i} \mu^i_k x_i^{N_i - n_i + k} y_i^{N_i - k} \in \CH^{2N_i - n_i}(\PP (S^2V_i) \times \PP (S^2 V_i^*)).
    \]
    Then 
    \[
        [J(Y_{L_1}, Y_{L_2})] =\left(\sum_{k=0}^{n_1} \mu_k^1 x^{N_1 - n_1 + k} y^{N_1 - k}\right) \cdot \left(\sum_{l=0}^{n_2} \mu_l^2 x^{N_2 - n_2 + l} y^{N_2 - l}\right)
    \]
    in $\CH^{2N-n-1}(\PP (S^2 V_1 \oplus S^2V_2) \times \PP (S^2 V_1^* \oplus S^2V_2^*))$.
\end{lemma}
\begin{proof}
    By Lemma~\ref{invariant_under_rational_equivalence}, we can replace $Y_{L_i}$ by a union of products of projective subspaces, with
    multiplicities indicated by the coefficients in the Chow ring. The class of the join is distributive over decomposing into irreducible components, so it is enough to understand class of join of two products of projective spaces. The join of $\PP^{n_1 - k} \times \PP^{k} \subset \PP(S^2 V_1) \times \PP(S^2 V_1^*)$ and $\PP^{n_2 - l} \times \PP^{l} \subset \PP(S^2 V_2) \times \PP(S^2 V_2^*)$ is the projective subspace $\PP^{n_1 + n_2 - k - l + 1} \times \PP^{k+l + 1} = \PP^{n-(k+l)} \times \PP^{k+l+1}$. Thus,
    \[
    \begin{split}
        [J(\PP^{n_1 - k} \times \PP^{k}, \PP^{n_2 - l} \times \PP^{l})] 
        &= [\PP^{n-(k+l)} \times \PP^{k+l+1}] 
        = x^{N - n + (k+l)} y^{N-(k+l +1)}\\
        &= x^{N_1 - n_1 + k} y^{N_1 - k} \cdot x^{N_2 - n_2 + l}y^{N_2 - l},
    \end{split}
    \]
    which implies that
    \[
    \begin{split}
        [J(Y_{L_1}, Y_{L_2})] 
        &= \sum_{k,l} \mu_k^1 \mu_l^2[J(\PP^{n_1 - k} \times \PP^{k}, \PP^{n_2 - l} \times \PP^{l})]\\
        &= \sum_{k,l} \mu_k^1 \mu_l^2 x^{N_1 - n_1 + k} y^{N_1 - k} \cdot x^{N_2 - n_2 + l}y^{N_2 - l}\\
        &= \left(\sum_{k=0}^{n_1} \mu_k^1 x^{N_1 - n_1 + k} y^{N_1 - k}\right) \cdot \left(\sum_{l=0}^{n_2} \mu_l^2 x^{N_2 - n_2 + l} y^{N_2 - l}\right).
    \end{split}
    \]
\end{proof}

%W dowodzie trzeba wyjaśnić, że używając lematu wyżej możemy zdeformować do sumy produktów przestrzeni rzutowych i zauważyć, że dla takich Join jest po prostu produktem przestrzeni rzutowych o wymiarach równych suma wym +1 Wtedy powołujemy się na to i w kolejnym lemacie (argumentując, że współczynnik będzie produktem dwóch liczb charakterystycznych) oraz w główynm twierdzeniu podsekcji wyznaczającym ciąg charakterystyczny dla sumy (pisząc że klasa Joina to produky klas).

\begin{lemma}
\label{l:join}
    For $i=1,2$ let $V_1, V_2$ be finite dimensional vector spaces and let $L_i \subset S^2V_i$ be vector subspaces of dimension $n_i+1$. Assume furthermore that $L_1$ and $L_2$ contain invertible matrices. Then in the Chow ring of $\PP (S^2 V_1 \oplus S^2V_2) \times \PP (S^2 V_1^* \oplus S^2V_2^*)$ we have
    \[ [Y_{L_1 \oplus L_2}] = [J(Y_{L_1}, Y_{L_2})] \cdot (x + y). \]
\end{lemma}
\begin{proof}
    After a suitable choice of coordinates, the map taking an $n \times n$ matrix to the tuple of its $((n-1) \times (n-1))$-minors is given by $A \mapsto \adj A$. Thus we can choose coordinates
    \[ \{x_{ij}\}, \{x'_{kl}\} \text{ on } \PP(S^2V_1 \oplus S^2V_2)\ \text{ and }\ \{y_{ij}\}, \{y'_{kl}\} \text{ on  } \PP(S^2V_1^* \oplus S^2V_2^*) \]
    so that the map above is given by 
    \[
        [A_1, A_2] \mapsto [\text{adj } A_1 \cdot \det A_2 , \det A_1 \cdot \text{adj } A_2]
    \]
    By assumption $L_1, L_2$ contain invertible matrices, so the locus of invertible matrices is 
    an open and dense subset $U \subset \PP(L_1 \oplus L_2)$. On $\CC$-points the variety $Y_{L_1 \oplus L_2}$ is defined by the equations in $x_{ij}, y_{ij}$ corresponding to equations of $Y_{L_1}$, and the equations in 
    $x'_{kl}, y'_{kl}$ corresponding to equations of $Y_{L_2}$ and an additional equation $f := \sum_j x_{1j} y_{j1} - \sum_{l} x'_{1l} y'_{l1} = 0$. The two first sets of equations define the join 
    $J(Y_{L_1}, Y_{L_2})$ of the varieties $Y_{L_1}, Y_{L_2}$. The additional equation $f$ corresponds to the normalization of the two factors on the open set $U$. That is, on $\CC$-points the variety $Y_{L_1 \oplus L_2}$ is given by
    \[
    \begin{split}
        Y_{L_1 \oplus L_2} &= J(Y_{L_1}, Y_{L_2}) \cap V(f) .
    \end{split}
    \]
    The variety $Y_{L_1 \oplus L_2}$ is irreducible by Lemma~\ref{l:irreducible}, so by~\cite[Theorem 1.26]{3264}, or \cite[Chapter 7]{Fulton}, there 
    exists an integer weight $w$ such that 
    \[
        [Y_{L_1 \oplus L_2}] = w\cdot [J(Y_{L_1}, Y_{L_2})] \cdot [V(f)] = w\cdot  [J(Y_{L_1}, Y_{L_2})] \cdot (x+y) .
    \]
    Since the characteristic number $c_{d,0,\dots,0}^L$ is $1$ for any $d+1$-dimensional linear space $L$, multiplying both sides by $x^{n_1+n_2+1}$ we get 
    \[
        w = w \cdot x^{n_1 + n_2 + 1} \cdot [Y_{L_1 \oplus L_2}] 
        = [J(Y_{L_1}, Y_{L_2})] \cdot (x^{n_1+n_2+2} +yx^{n_1+n_2+1})
        = [J(Y_{L_1}, Y_{L_2})] \cdot y x^{n_1+n_2+1},
    \]
    where the last equality follows from the vanishing of all higher powers of $x$ in the Chow ring. The last product is equal to the coefficient of $y^{N_1+N_2}$ in $[J(Y_{L_1}, Y_{L_2})]$. This coefficient is equal to the product of the two characteristic numbers $\mu^1_0 = c_{n_1,0,\dots,0}^{L_1}$ and $\mu^2_0 = c_{n_2,0,\dots,0}^{L_2}$. As above, both are equal to $1$, thus
    \[ w=c_{n_1,0,\dots,0}^{L_1} \cdot c_{n_2,0,\dots,0}^{L_2} = 1. \qedhere \]
\end{proof}

We are now ready to prove the recursive formula.
\begin{proposition}
    Let $V_1, V_2$ be vector spaces and let $L_1 \subset S^2V_1, L_2 \subset S^2V_2$ be subspaces of dimensions $n_1+1,n_2+1$ respectively, both containing invertible matrices. Let $\{\mu_m^{L_1\oplus L_2}\}, \{\mu_k^1\}$, and $\{\mu_l^2\}$ denote the characteristic sequences of $L_1 \oplus L_2, L_1$, and $L_2$, respectively. The characteristic sequences satisfy the following identity:
    \[
        \mu_m^{L_1\oplus L_2} = \sum_{k + l + 1 = m} \mu_k^1 \mu_l^2 + \sum_{k + l = m} \mu_k^1 \mu_l^2.
    \]
    for $m=0,\dots,n_1+n_2+1$.
\end{proposition}
\begin{proof}
    By combining Lemma~\ref{l:join} and Lemma~\ref{l:join_formula}, we get that
    \[
    \begin{split}
        [Y_{L_1 \oplus L_2}] 
        &= (x + y) \cdot \left(\sum_{k=0}^{n_1} \mu_k^1 x^{N_1 - n_1 + k} y^{N_1 - k}\right) \cdot \left(\sum_{l=0}^{n_2} \mu_l^2 x^{N_2 - n_2 + l} y^{N_2 - l}\right)\\
        &= \sum_{k=0}^{n_1} \sum_{l=0}^{n_2} \mu_k^1 \mu_l^2 x^{N - n + (k + l + 1)} y^{N - (k + l + 1)} + \sum_{k=0}^{n_1} \sum_{l=0}^{n_2} \mu_k^1 \mu_l^2 x^{N - n + (k + l)} y^{N - (k + l)}.
    \end{split}
    \]
    Thus
    \[
        \mu_m^{L_1\oplus L_2} = [Y_{L_1 \oplus L_2}] \cdot x^{n - m}y^m = \sum_{k + l + 1 = m} \mu_k^1 \mu_l^2 + \sum_{k + l = m} \mu_k^1 \mu_l^2. \qedhere
    \]
\end{proof}

\subsection{Characteristic numbers of explicit algebras}

Calculating characteristic numbers of explicit algebras is in general a difficult task.
In some cases we can give ad-hoc combinatorial arguments or utilize the geometry of the variety of complete quadrics. 

\subsubsection{The smooth algebra and $\CC[x]/(x^{n+1})$}\label{s:smooth}
Recall that an ($(n+1)$-dimensional) \emph{smooth algebra} is an algebra isomorphic to the product of $n+1$ copies of $\CC$. The characteristic numbers of the smooth algebra are known, but even in this seemingly simplest case the task of determining characteristic number was highly non-trivial. They corresponds to the mixed Eulerian numbers, which are an object of intensive study in combinatorics. We denote them by $e_{b_1, \dots, b_n}$. They are determined by the following relations \cite[Lemma 4.1]{mixedeuler}:
\begin{enumerate}
    \item $e_{1, 1, \dots, 1} = n!$,
    \item $2e_{a_1, \dots, a_n} = e_{a_1, \dots, a_{k-1}+1, a_k - 1, \dots, a_n} + e_{a_1, \dots, a_k - 1, a_{k+1} +1, \dots, a_n}$ if $1 < k < n$ and $a_k \geq 2$,
    \item $2e_{a_1, \dots, a_n} =  e_{a_1 - 1, a_2 + 1, a_3, \dots, a_n}$ if $a_1 \geq 2$,
    \item $2e_{a_1, \dots, a_n} = e_{a_1, \dots, a_{n-2}, a_{n-1}+1, a_n - 1}$ if $a_n \geq 2$.
\end{enumerate}
In particular, the characteristic sequence of the smooth algebra is 
\[
    \binom{n}{0}, \binom{n}{1}, \binom{n}{2}, \dots, \binom{n}{n-1}, \binom{n}{n}.
\]

In this section we prove that all characteristic numbers of the algebra $A_n = \mathbb{C}[x]/(x^{n+1})$ and the $(n+1)$-dimensional smooth algebra coincide. It is particularly interesting from the perspective of study of Hilbert schemes of points, as it shows that characteristic numbers are constant on $\operatorname{Hilb}_{n+1}(\mathbb{A}^1)$. \\

We start by showing that characteristic sequences of both algebras coincide. 
\begin{theorem}\label{(1,n)-numbers}
    For each $n \geq 0$, the characteristic sequences of $A_n = \CC[x]/(x^{n+1})$ and the smooth algebra $\prod_{i=1}^{n+1} \CC$ coincide and are given by
    \[
        \binom{n}{0}, \binom{n}{1}, \binom{n}{2}, \dots, \binom{n}{n-1}, \binom{n}{n}.
    \]
\end{theorem}
\begin{proof}
Recall that the characteristic sequence coincides with the bidegree of the graph $Y_n$ of the map defined in Example \ref{e:grad}. The multiplication table in $A_n$ is given by the matrix
\[
    M_n = 
    \begin{pmatrix}
    x_0 & x_1 & ... & x_{n-1} & x_n \\
    x_1 & ... & x_{n-1} & x_n & 0 \\
    ... & ... & ... & 0 & ... \\
    x_{n-1} & x_n & 0 & ... & 0 \\
    x_n & 0 & ... & 0 & 0 
    \end{pmatrix}. 
\]
Firstly, we show that $M_n$ has only $n+1$ pairwise distinct minors of size $n$. Note that $M_n$ is symmetric and, since any size $n$ minor of $M_n$ is obtained by deleting one row and one column from $M_n$, we can assume that index of a row we delete is at most the index of a column we delete for any minor. We denote the minor obtained by deleting the last row and the last column by $y_0$. For convenience, we introduce the following notation:

\[
    B_{n,k} :=  \begin{pmatrix}
    x_k & x_{k+1} & ... & x_{n-2} & x_{n-1} \\
    x_{k+1} & ... & x_{n-2} & x_{n-1} & x_n \\
    ... & ... & ... & x_n & 0 \\
    x_{n-2} & x_{n-1} & x_n & 0 & 0 \\
    x_{n-1} & x_n & 0 & 0 & 0 
    \end{pmatrix}, 
    \quad B_{n,n}:=1.
\]
We see that $y_0 = \det B_{n,0}$. Since $x_n$ is the only non-zero variable in the last row, all other minors are of the form $x_n\cdot m_{n-1}$, where $m_{n-1}$ is the minor of size $n-1$ of the matrix $M_{n-1}$ with variables $x_1,\ldots, x_{n}$ corresponding to the algebra $A_{n-1} = \mathbb{C}[x]/(x^n)$. Continuing inductively, we see that every minor of $M_n$ of size $n$ is equal (up to a sign) to $x_n^{k}\det B_{n,k}$ for an suitable $k$. Also note that $x_n^{k}\det B_{n,k}$ is equal (up to a sign) to the minor $y_k$ of $M_n$ obtained by deleting $(n-i)$-th row and $n$-th column from $M_n$. 
Therefore, the bidegree of $Y_n$ is equal to the bidegree of the graph $Y_n'$ of the map 
\[
     \PP^n \ni [x_0:\ldots:x_n] \mapsto [y_0,\ldots, y_n] \in \PP^n,
\]
where $y_i$ are defined as above (because these maps differ by a linear embedding $\PP^n \to \PP(S^2 A_n)$). \\

For $1 \leq k \leq n$, define 
\[
    f_k = \sum_{\substack{0 \leq i,j \leq n \\ i+j=2n-k}}(-1)^ix_iy_j.
\]
We claim that $V(f_1,\ldots, f_n) = Y_n'$. We have $f_n \in I(Y_n')$, since on the open subset $\{x_n \neq 0\}$ it is the determinant of a degenerate matrix, as evident by replacing the last column of $M_n$ by $[x_0 \:x_1\:\ldots\:x_n]^{\top}$ and taking the Laplace expansion along this column. Now, we show that $(f_1,\ldots,f_{n-1}) \in I(Y_{n-1}')$ by induction on $n$. The case $n=1$ follows from the expansion above. For any $k$ we have 
\[
    f_k = \sum_{\substack{0 \leq i,j \leq n \\ i+j=2n-k}}(-1)^ix_iy_j = x_n^{n-k}\sum_{\substack{0 \leq i,j \leq n \\ i+j=2n-k}}(-1)^ix_n\frac{y_j}{x_n^{n-k}}
\]
and 
\[ 
    \frac{y_j}{x_n^{n-k}} = y_{j-(n-k)}^k,
\]
where $y_{j-(n-k)}^k$ is a corresponding minor $y_{j-(n-k)}$ for $M_k$ with variables $x_{n-k},\ldots, x_n$ instead of $x_0,\ldots, x_k$. Therefore, $f_k \in I(Y_k')$ follows from the $k$-th induction step. Therefore, $Y_n' \subseteq V(f_1,\ldots,f_n)$. Moreover, both $Y_n'$ and $V(f_1,\ldots,f_n)$ are irreducible and has the same codimension, so in fact $Y_n' = V(f_1,\ldots,f_n)$. \\

Recall that the Chow ring of $\PP^n \times \PP^n$ is isomorphic to
\[
    \CH(\PP^n \times \PP^n) \simeq \ZZ[x]/(x^{n}) \otimes \ZZ[y]/(y^{n}),
\]
where $x, y$ correspond to hyperplanes on the respective copies of $\PP^n$.
Each $f_k$ is a degree $(1,1)$ polynomials, so we have $[V(f_k)] = x+y$ in $\CH(\mathbb{P}^n \times \mathbb{P}^n)$. The divisors $V(f_1),\ldots, V(f_n)$ intersect transversally at a general point of $V(f_1,\ldots, f_n)  = \bigcap_{k=1}^n V(f_k)$ since they are transversal at $([0:\ldots:0:1],[0:\ldots:0:1]) \in X$, so 
\[
    [Y_n] = [Y_n'] = (x+y)^n = \sum_{k=0}^n\binom{n}{n-k}x^k y^{n-k},
\]
which shows that the $k$-term of the characteristic sequence for $A_n$ is $\binom{n}{k}$.
\end{proof} \label{4.4}

We prove that all the characteristic numbers of the smooth algebra and $A_n$ coincide. The idea of the proof is to interpret the calculation of characteristic numbers as calculations in the Chow ring of the variety of complete quadrics and use the relations between different bases of the group of divisors to obtain recursive formulas, which coincide with the recursive formulas determining the characteristic numbers of the smooth algebra. \\

%he space of divisors $\Pic_{\mathbb{Q}}(\mathcal{CQ}(V))$ has two distinguished bases, denoted by $\{L_i\}$ and $\{S_i\}$, respectively. Recall that they are related by the following relations,:
%\[
%	S_i = - L_{i-1} + 2 L_i - L_{i+1},
%\]
%where we set $L_0 = L_n := 0$. We will use those relations in our computations 
%The first ones are pullbacks of hyperplanes in $\PP(\bigwedge^i S^2 V)$ and the second ones are strict transforms of the exceptional loci of the blow-ups $X_{i-1} \leftarrow X_i$.  \\
%We denote $\mathrm{Gr}(k,V) := \mathrm{Gr}(k,V)$ when the space $V$ is clear from the context.

We introduce the relative variety of complete quadrics. The main idea of this definition is to view the orbits of $\mathcal{CQ}(V) = \mathcal{LG}(V)/\mathbb{C}^*$ that break at the $k$-th Grassmannian $\mathrm{Gr}(k,V) \hookrightarrow \mathcal{CQ}(V)$ as a pair of compatible smaller orbits, one connecting the unique point of $\mathrm{Gr}(0,V)$ to a point $U$ of $\mathrm{Gr}(k,V)$ and another connecting $U$ to the unique point of $\mathrm{Gr}(n,V)$.

\begin{definition}
        Let $V$ be a $n$-dimensional vector space and consider the rational map
    \[
        (\Phi_{\mathcal{Q}})_k: \mathbb{P}(S^2V) \dashrightarrow \mathbb{P}(S^2V) \times \mathbb{P}\left(S^2\left(\bigwedge^2V\right)\right) \times \ldots \times \mathbb{P}\left(S^2\left(\bigwedge^{k}V\right)\right)
    \]
    which is defined to agree with $\Phi_{\mathcal{Q}}$ from Definition \ref{cq} on the first $k$ coordinates of the co-domain.
We define the subspace-relative variety of complete quadrics $\mathcal{CQ}_k(\mathrm{Gr}(k,V))$ as the closure of the image of the set of rank $k$ matrices under this map.    
\end{definition}

The space $\mathcal{CQ}_k(\mathrm{Gr}(k,V))$ is a fiber bundle over the Grassmannian $\mathrm{Gr}(k,V)$ with a natural morphism given by $$\varphi_k: \mathcal{CQ}_k(\mathrm{Gr}(k,V)) \ni [w_1 \wedge \ldots \wedge w_k] \mapsto \langle w_1,\ldots,w_k\rangle \in \mathrm{Gr}(k,V).$$ Note that $\varphi_k$ is well-defined on the open set of rank $k$ matrices, and, since $w_k$ is a rank $1$ tensor on this set, the same is true for the closure, hence $\varphi_k$ is well-defined. The fiber over $W \in \mathrm{Gr}(k,V)$ is isomorphic to $\mathcal{CQ}(W)$.

The interpretation of $\mathcal{CQ}(V)$ as the Chow quotient $\mathcal{LG}(V)/\mathbb{C}^*$ provides an alternative description of $\mathcal{CQ}_k(\mathrm{Gr}(k,V))$. Namely, the $\mathbb{C}$-points of $\mathcal{CQ}_k(\mathrm{Gr}(k,V))$ correspond to the orbits between $V \subset V\oplus V^{*}$ and $\mathrm{Gr}(k,V) \hookrightarrow \mathcal{CQ}(V)$. In other words, the space $\mathcal{CQ}_k(\mathrm{Gr}(k,V))$ parametrizes left parts (that is, the restrictions to the part between $V$ and $\mathrm{Gr}(k,V)$) of orbits in $\mathcal{CQ}(V)$  that break at $\mathrm{Gr}(k,V)$. 
Therefore, the image of an orbit $O$ under the projection $\varphi_k: \mathcal{CQ}_k(\mathrm{Gr}(k,V)) \to \mathrm{Gr}(k,V)$ is the intersection of $O$ with $\mathrm{Gr}(k,V)$.

Note that the isomorphism class of $\mathcal{CQ}_k(\mathrm{Gr}(k,V))$ does not depend on the choice of an $n$-dimensional space $V$, as any isomorphism $V \simeq W$ induces the isomorphism of the varieties above.

%Now, we can define the dual version of the subspace-relative variety of complete quadrics.

%\begin{definition}
 %  Let $V$ be a $n$-dimensional vector space and consider the rational map \[
  %      (\Phi_{\mathcal{Q}})^k: \mathbb{P}(S^2V) \dashrightarrow \mathbb{P}\left(S^2\left(\bigwedge^{k}V\right)\right) \times \mathbb{P}\left(S^{2}\left(\bigwedge^{k+1}V\right)\right) \times \ldots \times \mathbb{P}\left(S^2\left(\bigwedge^{n-1}V\right)\right)
   % \]
    %which is the restriction of the map $\Phi_{\mathcal{Q}}$ in Definition \ref{cq} to the last $n-k$ coordinates of the co-domain.
%We define the \emph{quotient-relative variety of complete quadrics} $\mathcal{CQ}^k(\mathrm{Gr}(k,V))$ as the closure of the image of the set of rank $k$ matrices under this map. 
%\end{definition}

%Therefore, the space $\mathcal{CQ}^k(\mathrm{Gr}(k,V))$ is a fiber bundle over $\mathrm{Gr}(k,V)$ whose fiber over $W \in \mathrm{Gr}(k,V)$ is isomorphic to $\mathcal{CQ}(V/W)$. 

\begin{lemma}
    Let $V$ be an $n$-dimensional vector space. There is an isomorphism
    \[\mathcal{CQ}(V) \cap S_k \simeq \mathcal{CQ}_k(\mathrm{Gr}(k,V)) \times_{\mathrm{Gr}(k,V)} \mathcal{CQ}_{n-k}(\mathrm{Gr}(n-k,V^*)). \]
\end{lemma}
\begin{proof}
    It follows from the definition of $\mathcal{CQ}_k(\mathrm{Gr}(k,V))$ that the projection of $\mathcal{CQ}(V) \cap S_k$ onto the first $k$ coordinates $\mathbb{P}(S^2V) \times \ldots \times \mathbb{P}\left(S^2\left(\bigwedge^{k}V\right)\right) $ is isomorphic to $\mathcal{CQ}_k(\mathrm{Gr}(k,V))$.
    There is an isomorphism $\mathcal{LG}(V) \rightarrow \mathcal{LG}(V^*)$ given by viewing a subspace $W \subset V \oplus V^*$ as the same subspace $W \subset V^* \oplus (V^*)^* \simeq V^* \oplus V$. This induces an isomorphism $\alpha: \mathcal{CQ}(V) \rightarrow \mathcal{CQ}(V^*)$. By the definition of $S_k$ as
    \[ S_k = \{(A_1, \dots, A_{n-1}) \in \mathcal{CQ}(V) \colon \mathrm{rk}(A_k) = 1\}, \]
    we see that $\alpha$ maps $\mathcal{CQ}(V) \cap S_k$ to $\mathcal{CQ}(V^*) \cap S_{n-k}$. Under $\alpha$ the coordinate $\mathbb{P}\left(S^2\left(\bigwedge^{j}V\right)\right)$ is interchanged with $\mathbb{P}\left(S^2\left(\bigwedge^{n-k}V^*\right)\right)$, so the image of $\mathcal{CQ}(V) \cap S_k$ under the projection to the last $n-k$ coordinates is isomorphic to
    \[ \mathcal{CQ}_{n-k}(\mathrm{Gr}(n-k, V^*)). \]
    Thus $\mathcal{CQ}(V) \cap S_k$ is a closed subvariety of the fibre product
    \[ \mathcal{CQ}_k(\mathrm{Gr}(k,V)) \times_{\mathrm{Gr}(k,V)} \mathcal{CQ}_{n-k}(\mathrm{Gr}(n-k,V^*)), \]
    with the map to $\mathrm{Gr}(k,V)$ given by the projection to the $k$-th coordinate on $\mathcal{CQ}_k(\mathrm{Gr}(k,V))$, and the projection to the $(n-k)$-th coordinate composed with an isomorphism $\mathrm{Gr}(n-k,V^*) \simeq \mathrm{Gr}(k,V)$ on $\mathcal{CQ}_k(\mathrm{Gr}(k,V))$. We see that
    \begin{gather*}
        \dim\left( \mathcal{CQ}_k(\mathrm{Gr}(k,V)) \times_{\mathrm{Gr}(k,V)} \mathcal{CQ}_{n-k}(\mathrm{Gr}(k,V))\right) \\
    = \dim(\mathcal{CQ}_k(\mathrm{Gr}(k,V)) + \dim(\mathcal{CQ}_{n-k}(\mathrm{Gr}(k,V)) - \dim(Gr(k,V))  \\
    = \left( \frac{k\cdot (k+1)}{2}-1 + k\cdot (n-k) \right) + \left( \frac{(n-k)\cdot (n-k+1)}{2} -1 + k\cdot (n-k) \right) - k\cdot (n-k) \\
    = \frac{n\cdot (n+1)}{2} -2 = \dim(\mathcal{CQ}(V)) -1 = \dim(\mathcal{CQ}(V) \cap S_k).
    \end{gather*}
    Thus the dimensions agree, so $\mathcal{CQ}(V) \cap S_k$ cannot be a proper closed subvariety.
\end{proof}
Recall that under the interpretation of $\mathcal{CQ}(V) \cap S_k$ consists of orbits of the $\mathbb{C}^*$-action on $\mathcal{LG}(V)$ that break at the $k$-th embedded Grassmannian $\mathrm{Gr}(k,V) \hookrightarrow \mathcal{CQ}(V)$, hence there is a natural map $\mathcal{CQ}(V) \cap S_k \to \mathrm{Gr}(k,V)$.
%It follows from the discussion above that there is a natural isomorphism between and $\mathcal{CQ}_k(\mathrm{Gr}(k,V)) \times_{\mathrm{Gr}(k,V)} \mathcal{CQ}_{n-k}(\mathrm{Gr}(k,V))$ given by the projection to the $k$ first coordinates and to the $n-k$ last coordinates. \\

\begin{lemma}\label{translatingdivisors}
    Let $V$ be an $n$-dimensional vector space.
    Under the identification $\mathcal{CQ}(V) \cap S_k$ with $\mathcal{CQ}_k(\mathrm{Gr}(k,V)) \times_{\mathrm{Gr}(k,V)} \mathcal{CQ}_{n-k}(\mathrm{Gr}(k,V))$, the divisor $L_i$ on $\mathcal{CQ}(V)$ corresponds to $L_i$ on $\mathcal{CQ}_{k}(\mathrm{Gr}(k,V))$ for $i < k$ and  to $L_{i-k}$ on $\mathcal{CQ}_{n-k}(\mathrm{Gr}(k,V))$ for $i > k$.
\end{lemma}

\begin{proof}
    Recall that the isomorphism $\mathcal{CQ}(V) \cap S_k \simeq \mathcal{CQ}_k(\mathrm{Gr}(k,V)) \times_{\mathrm{Gr}(k,V)} \mathcal{CQ}_{n-k}(\mathrm{Gr}(k,V))$ is  given by the projection to the $k$ first coordinates and to the $n-k$ last coordinates.
    Thus, for $i < k$, the claim follows from the construction of $\mathcal{CQ}_k(\mathrm{Gr}(k,V))$ and definition of $L_i$.
    By symmetry (applied twice, once to the smaller and once to the bigger $\mathcal{CQ}$), we see that for $i > k$ the divisor $L_i$ on $\mathcal{CQ}(V)$ is identified with $L_{(n-k)-(n-i)}=L_{i-k}$ on $\mathcal{CQ}_{n-k}(\mathrm{Gr}(k,V))$.
\end{proof}

For an algebra $A$ of dimension $n+1$, we define the subvariety $X_A \subset \mathcal{CQ}(A^*)$. The characteristic number $c^A_{b_1, \dots, b_{n}}$ coincides with the element
\[
	[X_A] \cdot L_1^{b_1} \cdot L_2^{b_2} \cdot \dots \cdot L_{n}^{b_{n}} \in \CH^{n+1}(\mathcal{CQ}(A^*)) \cong \ZZ.
\]
To simplify the notation, we denote the variety $X_{A_n}$ as $X_n$ and the characteristic number $c^A_{b_1, \dots, b_{n-1}}$ as $c_{b_1, \dots, b_{n-1}}$. \\

\begin{lemma}\label{(n-1, 0, ..., 1, 0)}
    The characteristic number $c_{n-1, 0, \dots, 1, 0}$ is equal to $n-1$.
\end{lemma}
%using their explicit form given in the proof of \ref{(1,n)-numbers}
\begin{proof}
    Take $n-1$ hyperplanes given by general linear combinations of $x_0, x_1, \dots, x_n$, that is, $ a_{i0} x_0 + \dots + a_{in} x_n$ for $i=1,\dots,n-1$. We may assume that the matrix $[a_{ij}]_{1 \leq i\leq n-1,0\leq j \leq n-2}$ is invertible. On the intersection of those hyperplanes the variables $x_0, \dots, x_{n-2}$ can be expressed as general linear combinations of  $x_{n-1}, x_n$. Then by induction on $n$ we can see that on this intersection the $(n-1)$-minors of $M_n$ are general linear combinations of all degree $n-1$ monomials in $x_{n-1}, x_n$.
    Therefore the map $\PP^1 = \PP(S^2(A^*))|_H \to \PP(S^2(\bigwedge^{n-1} A^*))$, whose degree by definition coincides with the characteristic number $c_{n-1, 0, \dots, 1, 0}$, is the composition of the $(n-1)$-th Veronese embedding $v_{n-1}$ of $\PP^1$ with a linear embedding $\PP^{n-1} \to \PP(S^2(\bigwedge^{n-1} A^*))$. Thus the characteristic number $c_{n-1, 0, \dots, 1, 0}$ is equal to $n-1$, the degree of $v_{n-1}(\PP^1)$.
\end{proof}

\begin{lemma}\label{intersecting with S_k set-theorethically}
    The intersection of $X_{n+1} \cap S_k$ agrees with $X_k \times X_{n+1-k}$ on $\CC$-points.
\end{lemma}
\begin{proof}
    By construction of $\mathcal{CQ}(V)$ as the Chow quotient of $\mathcal{LG}(V)$, the $\CC$-points of $X_{n+1} \cap S_k$ correspond exactly to the broken orbits intersecting the $k$-th embedded Grassmannian. The only possible linear span of rank $k$ matrices in the subspace defined by our algebra is the coordinate subspace $V_k$ spanned by first $k$ basis vectors, so these orbits have to break at this point. It is enough to consider our claim on the open set of orbits containing this point. Such an orbit is given by a pair of orbits from the 0-point to $V_k$ and from $V_k$ to $V$. The first orbit corresponds exactly to a point in $X_k$, as the rank $\leq k$ matrices in the linear subspace are exactly those with coordinates $x_k,...,x_n$ equal to zero and the subspace corresponding to $X_k$ is exactly the subspace of rank $\leq k$ matrices in the subspace corresponding to $X_{n+1}$ under the inclusion $V_k \subset V$. The algebra $A_n$ is Gorenstein, so by interchanging $V$ with with $V^*$ we see by symmetry that the elements of the second orbit correspond exactly to the elements of $X_{n+1-k}$. Thus the $\CC$-points $X_{n+1} \cap S_k$ can be interpreted as $\CC$-points of $X_k \times X_{n+1-k}$. By dimensional reasons it cannot be a proper subvariety, which concludes the proof.
\end{proof}

\begin{lemma}\label{intersecting with S_k in Chow ring}
    Classes of $X_{n+1} \cap S_k$ and $X_k \times X_{n+1-k}$ are related by the following relation:
    \[
        [X_{n+1}] \cdot S_k = {n+1 \choose k} \cdot [X_{k} \times X_{n+1-k}].
    \]
\end{lemma}
\begin{proof}
    Recall that each $X_i$ is irreducible by Lemma~\ref{l:irreducible} an $\mathcal{CQ}(V)$ is smooth variety. By Lemma \ref{intersecting with S_k set-theorethically} 
    combined with~\cite[Theorem 1.26]{3264} there exist weights $w_{nk}$ such that $[X_{n+1}] \cdot S_k = w_{nk}\cdot [X_{k} \times X_{n+1-k}]$.
    We first determine the value of $w_{nn}$. Using Lemma \ref{translatingdivisors} we get
    \[
    \begin{split}
        w_{nn} 
        &= w_{nn} \cdot [X_n \times X_1] L_1^{n-1}
        = [X_{n+1}] \cdot S_n \cdot L_1^{n-1} \\
        &= [X_{n+1}] \cdot (2 L_n - L_{n-1}) \cdot L_1^{n-1} 
        = 2 c_{n-1, 0, \dots, 0, 1} - c_{n-1, 0, \dots, 1, 0}.
    \end{split}
    \]
    Hence by Theorem \ref{(1,n)-numbers} and Lemma \ref{(n-1, 0, ..., 1, 0)} we obtain $w_{nn} = 2n - (n-1) = n +1$.\\
    Now we determine the values of all $w_{nk}$. By iterating the previous result we get that 
    \[
        [X_{n+1}] \cdot S_n \cdot S_{n-1} \cdot \dots \cdot S_k = \frac{(n+1)!}{k!} \cdot [X_k \times X_1 \times \dots \times X_1].
    \]
    By Lemma \ref{translatingdivisors} and the formula relating $S_i$ with $L_i$ we also get
    \[
    \begin{split}
        [X_{n+1}] \cdot S_k \cdot S_n \cdot S_{n-1} \cdot \dots \cdot S_{k+1}
        &= w_{nk} \cdot [X_{k} \times X_{n+1-k}] \cdot S_{n} \cdot S_{n-1} \cdot \dots \cdot S_{k+1} \\
        &= (n+1-k) \cdot w_{nk} \cdot [X_k \times X_{n-k} \times X_1] S_{n-1} \cdot \dots \cdot S_{k+1}
        = \dots \\
        &= (n+1-k)! \cdot w_{nk} \cdot [X_k \times X_1 \times \dots \times X_1],
    \end{split}
     \]
     so we conclude that
     \[
        w_{nk} = \frac{(n+1)}{k! (n+1-k)!} = {n+1 \choose k}. \qedhere
     \]
\end{proof}

\begin{theorem}\label{t:numbers}
    The characteristic numbers of the smooth algebra $\prod_{k=1}^{n+1} \CC$ and $A_n = \CC[x]/(x^{n+1})$ coincide.
\end{theorem}
\begin{proof}
    Let $b_1, \dots, b_n$ be non-negative integers satisfying $\sum_{i=1}^n b_i = n-1$. Lemmas \ref{intersecting with S_k in Chow ring} and \ref{translatingdivisors} imply that
    \[
        [X_{n+1}] \cdot S_k \cdot L_1^{b_1} \cdot \dots \cdot L_{n}^{b_{n}} = 
        \begin{cases}
            {n+1 \choose k} \cdot c_{b_1, \dots, b_{k-1}} \cdot c_{b_{k+1}, \dots, b_{n}} & \quad\text{if } b_k = 0 \text{ and } \sum_{i=1}^{k-1} b_i = k-1 \\
            0 & \quad \text{otherwise} \\
        \end{cases}
    \]
    Recall that $e_{a_1, \dots, a_n}$ denote the mixed Eulerian numbers, that is, the characteristic numbers of the smooth algebra. We show that $c_{a_1, \dots, a_n}$ satisfy the recursive relations uniquely determining $e_{a_1, \dots, a_{n}}$.
    %TODO maybe introduce notation c_{1^n}
    We proceed by induction, for a related approach see \cite[Proposition 4.4]{LLT}. For $n = 0$ we have $c_{1^0} = 1 = 0!$ by direct inspection. Write $L_1$ as 
    \[
    	L_1 = \sum_k \lambda_k S_k.
    \] 
    Note that $\lambda_k$ form the first row of the matrix
    \[
    \begin{bmatrix}
        2 & - 1 & 0 & 0 & \dots & 0 & 0 \\
        -1 & 2 & -1 & 0 & \dots & 0 & 0 \\
        0 & -1 & 2 & -1 & \dots & 0 & 0 \\
        0 & 0 & -1 & 2 & \dots & 0 & 0 \\
        \vdots & \vdots & \vdots & \vdots & \ddots & \vdots & \vdots \\
        0 & 0 & 0 & 0 & \dots & 2 & -1 \\
        0 & 0 & 0 & 0 & \dots & -1 & 2 \\
    \end{bmatrix}^{-1},
    \]
    relating the two bases $\{L_i\}, \{S_i\}$. Direct calculation yields that $\lambda_{1} = n / (n+1)$. Therefore we get
    \[
    \begin{split}
        c_{1^n}
        &= [X_{n+1}] \cdot L_1 \cdot L_2 \cdot \dots \cdot L_n 
        = \sum_k \lambda_k [X_{n+1}] \cdot S_k \cdot L_2 \cdot \dots \cdot L_n \\
        &= \lambda_1 [X_{n+1}] \cdot S_1 \cdot L_2 \cdot \dots \cdot L_n
        = \lambda_1 {n+1 \choose 1} \cdot c_{1^{0}}\cdot  c_{1^{n-1}}\\
        &= \frac{n}{n+1} \cdot (n+1) \cdot 0! \cdot (n-1)!
        = n!
    \end{split}
    \]
    
    Now we prove the other claims. First assume that $1 < k < n$ and $a_k \geq 2$. Define 
    \[
        b_j = \begin{cases}
            a_j & \quad\text{if } j \neq k\\
            a_{k} - 1 & \quad\text{if } j = k.
        \end{cases}
    \]
    We have $b_k \neq 0$, so
    \[
    \begin{split}
        0 &= [X_{n+1}] \cdot S_k \cdot L_1^{b_1} \cdot \dots \cdot L_{n}^{b_{n}} 
        = [X_{n+1}] \cdot (2 L_k - L_{k-1} - L_{k+1}) \cdot L_1^{b_1} \cdot \dots \cdot L_{n}^{b_{n}} \\
        &= 2c_{a_1, \dots, a_n} - c_{a_1, \dots, a_{k-1}+1, a_k - 1, \dots, a_n} - c_{a_1, \dots, a_k - 1, a_{k+1} +1, \dots, a_n}.
    \end{split}
    \]
    This shows that the second relation holds. Analogous argument shows that the third and fourth relation holds as well (they are just degenerate cases of the second relation).
\end{proof}

\subsubsection{The CW tensor}\label{s:cw}

Consider the algebra 
\[
    CW_n = \CC[x_0, x_1, \dots, x_{n-2}]/(\alpha_0^2 + \alpha_1^2 + \dots + \alpha_{n-2}^2)^\perp,
\]
where $\alpha_i$ is dual to $x_i$ and $(-)^\perp$ denotes the apolar ideal. The tensor associated to $CW_n$
is called the \emph{Coppersmith-Winograd tensor}. It has important applications in complexity theory 
as it is the main building block of the algorithms proving to be the best known bounds on the exponent of matrix multiplication $\omega$, see \cite{MR4262465}. In this section, we calculated the characteristic numbers of $CW_n$. \\

We start by giving an explicit description of the space spanned by $(k \times k)$-minors of the space of matrices associated to $CW_n$.
\begin{lemma}\label{l:minors_cw}
    The space of $(k \times k)$-minors of the space of matrices associated to $CW_n$ has the following basis:
    \[
    \begin{cases}
        \{x_i \colon 0 \leq i \leq n\}  & \text{if } k = 1 \\
        x_n^{k-2} \cdot (\{x_i x_j \colon 0 < i, j \leq n, i + j \neq n\} \cup  \{k x_i x_{n-i} - x_0 x_n \colon 0 < i \leq n/2\}) & \text{if }  1 < k < n \\
        x_n^{n-2} \cdot (\{x_i x_n \colon 0 < i \leq n\} \cup \{ \sum_{i=1}^{n-1} x_i x_{n-i} - x_0 x_n\}) & \text{if } k = n \\
    \end{cases}
    \]
\end{lemma}
\begin{proof}
The space of matrices associated to $CW_n$ is
\[
\begin{bmatrix}
    x_0 & x_1 & x_2 & \dots & x_{n-2} & x_{n-1} & x_n \\
    x_1 & 0 & 0 & \dots & 0 & x_n & 0 \\
    x_2 & 0 & 0 & \dots & x_n & 0 & 0 \\
    \vdots & \vdots & \vdots & \ddots & \vdots & \vdots & \vdots \\
    x_{n-2} & 0 & x_n & \dots & 0 & 0 & 0 \\
    x_{n-1} & x_n & 0 & \dots & 0 & 0 & 0 \\
    x_n & 0 & 0 & \dots& 0 & 0 & 0 \\
\end{bmatrix}.
\]
We proceed by induction. We index rows and columns starting from 0. Minors of size $k$ are indexed by $k$-tuples of sorted indices $i_1 < \dots < i_k, j_1 < \dots < j_k$. The case $k = 1 $ is trivial, so assume that $k > 1$.

First, we consider the case $i_1 = j_1 = 0$. Take the $k \times k$-submatrix $M$ obtained by choosing rows and columns indexed by $i_1 < \dots i_k, j_1 < \dots j_k$. The first column of $M$ is $x_0, x_{i_2} \dots, x_{i_k}$ and the first row of $M$ is $x_0, x_{j_2}, \dots, x_{j_k}$. For any other indices $2 \leq s, t \leq k$, the $(i_s,j_t)$-th entry is $x_n$ if $s + t = n$ and is zero otherwise. If $M_{2 \leq s, t \leq k}$ has $k-1$ non-zero entries, that is, it has $x_n$ on the antidiagonal and zeroes everywhere else, then a direct calculation using Laplace expansion shows that $\det M = x_n^{k-2} \cdot (\sum_{s = 2}^k x_{i_s} x_{n - i_s} - x_0 x_n)$. If $M_{2 \leq s, t \leq k}$ has $k-2$ non-zero entries, that is, it becomes an $(k-2) \times (k-2)$-matrix with $x_n$ on the antidiagonal and zeroes everywhere else after crossing out one row $i_s$ and one column $j_t$, then a direct calculation shows that $\det M = x_n^{k-2}\cdot x_{i_s} x_{j_t}$. Otherwise, $M_{2 \leq s, t \leq k}$ has rank at most $k-3$, so by subadditivity of matrix rank one obtains $\operatorname{rk} M \leq 2 + (k-3) < k$, which implies that $\det M = 0$.

Now, consider the case $i_1 = 0, j_1 > 0$ (which also covers the case $i_1 > 0, j_1 = 0$ by symmetry).
We obtain a matrix $M$ whose first row is $x_{j_1}, \dots, x_{j_k}$, each other row is zero or contains exactly one non-zero entry equal to $x_n$ and each column contains at most one $x_n$. If there exists a zero row, then $\det M = 0$. Otherwise let $j_t$ be the index such that the column indexed by $x_{j_t}$ does not contain $x_n$. Taking the Laplace expansion with respect to this column yields $\det M = x_n^{k-1} \cdot x_{j_t}$.

Finally, if $i_1, j_1 > 0$, then we can clearly obtain only $x_n^k$ or zero. \\

Note that for $k < n$, then the span of all minors of the form $x_n^{k-2} \cdot (\sum_{s = 2}^k x_{i_s} x_{n - i_s} - x_0 x_n)$ has basis $x_n^{k-2} \cdot \{k x_i x_{n-i} - x_0 x_n \colon 0 < i \leq n/2\}$, and for $k = n$ the only minor of this form is $x_n^{k-2} \cdot(\sum_{i=1}^{n-1} x_i x_{n-i} - x_0 x_n)$. A direct verification shows that the other forms of minors indeed do appear, so we obtain basis as stated in the formulation of this lemma.
\end{proof}

Now we use this explicit form of minors to determine the characteristic numbers of $CW_n$. 

\begin{proposition}\label{e:cw}
    The characteristic numbers of $CW_n$ are given by:
    \[
        c_{a_1, \dots, a_n} = 
        \begin{cases}
            2^{a_2 + \dots + a_{n-1} - 1} & \text{if } a_1 = a_n = 0, \\
            2^{a_2 + \dots + a_{n-1}} & \text{if } a_1 = 0, a_n > 0 \text{ or } a_1 > 0, a_n = 0, \\
            2^{a_2 + \dots + a_{n-1} + 1} & \text{if } a_1, a_n > 0. \\
        \end{cases}
    \]
\end{proposition}
\begin{proof}
	The characteristic numbers of $CW_n$ can be calculated by taking the rational map defined the minors from Lemma \ref{l:minors_cw} and calculating the multidegree of its graph. It is sufficient to consider the restriction of this map to the open subset $\{x_n \neq 0 \}$. There, it is given by
	\[
    \begin{split}
        [x_0 : \dots : x_n] \mapsto & \prod_{k=2}^{n-1} [\{x_n^{k-2}x_i x_j\} : \{x_n^{k-2}(kx_ix_{n-i} - x_0x_n)\}] \times [\{x_i x_n^{n-1}\} : x_n^{n-2}(\sum_{i=1}^{n-1} x_i x_{n-i} - x_0 x_n)] \\
        = & \prod_{k=2}^{n-1} [\{x_i x_j\} : \{kx_ix_{n-i} - x_0x_n\}] \times [\{x_i x_n\} : \sum_{i=1}^{n-1} x_i x_{n-i} - x_0 x_n],
    \end{split}
	\] 
	where the range of indices is as in Lemma \ref{l:minors_cw}.
	
	Consider the case $a_1 > 0$. Choose a general hyperplane on the domain given by a linear equation with a non-zero coefficient by $x_0$, say $x_0 = \sum_{i=1}^n a_i x_i$. After restricting to this hyperplane, the variable $x_0$ can be expressed as a linear combination of $x_1, \dots, x_n$, so the map can be described as 
    \[
        [x_1 : \dots : x_n] \mapsto \prod_{k=2}^{n-1} [\{x_i x_j\} : \{kx_ix_{n-i} - \sum_{i=1}^n a_i x_ix_n\}] \times [\{x_i x_n\} : \sum_{i=1}^{n-1} x_i x_{n-i} - \sum_{i=1}^n a_i x_i x_n].
    \]
    The linear span of $\{x_i x_j\}$ and $\{kx_ix_{n-i} - \sum_{i=1}^n a_i x_ix_n\}$ has basis $\{x_i x_j \colon 1 \leq i, j \leq n \}$
    and the linear span of $\{x_i x_n\}$ and $\sum_{i=1}^{n-1} x_i x_{n-i} - \sum_{i=1}^n a_i x_i x_n$ has basis $\{x_i x_n \colon 1 \leq i \leq n\} \cup \{ \sum_{i=1}^{n-1} x_i x_{n-i}\}$, so we can instead use the map
	\[
		[x_1 : \dots : x_n] \mapsto \prod_{k=2}^{n-1} [\{x_i x_j \colon 1 \leq i, j \leq n \}] \times [\{x_i x_n \colon 1 \leq i \leq n\}: \sum_{i=1}^{n-1} x_i x_{n-i}],
	\]
    as it differs from the original map by a linear embedding of smaller projective spaces into larger projective spaces on the codomain.
    
	Note that for the first $n-1$ factors of the codomain, the composition of this map with the projection onto this factors coincides with the second Veronese embedding $v_2$ of $\PP^{n-1}$. In particular, if $a_n = 0$, then $c_{a_1, a_2, \dots, a_{n-1}, 0}$ is the multidegree of the graph of the diagonal map $x \mapsto \prod_{k=2}^{n-1} v_2(x)$, so 
    \[
        c_{a_1, a_2, \dots, a_{n-1}, 0} = 2^{a_2 + \dots + a_{n-1}}.
    \]
    The algebra $CW_n$ is Gorenstein, so by Proposition \ref{p:gorenstein}, the characteristic numbers are symmetric and thus this also settles the case $a_1 = 0, a_n > 0$.
	
Now consider the case when $a_n > 0$ (and still $a_1 > 0$). Take $a_n$ general hypersurfaces given by linear combinations of $x_i x_n$ and $\sum_{i=1}^{n-1} x_i x_{n-i}$. We can assume that the first equation has a non-zero coefficient by $\sum_{i=1}^{n-1} x_i x_{n-i}$ and the others have coefficient zero. After restricting to the open subset $\{x_n \neq 0 \}$, the last $a_n - 1$ hypersufraces become hyperplanes given by general linear combinations of $x_1, \dots, x_n$. Thus, by B{\'e}zout's theorem, intersecting with $a_n$ general hyperplanes on the last factor contributes a factor of $2$ to the multidegree, so 
\[
    c_{a_1, \dots, a_n} = 2^{a_2 + \dots + a_{n-1} + 1}.
\]

We are left with the case $a_1 = a_n = 0$. We want to calculate the characteristic numbers of the variety coming from the rational map
\[
	[x_0: \dots : x_n] \mapsto \prod_{k=2}^{n-1} [\{x_i x_j \colon 1 \leq i, j \leq n, i + j \neq n\} : \{k x_i x_{n-i} - x_0 x_n \colon 0 < i \leq n/2\}].
\]
Choose an $a_k > 0$ and take a hypersurface given by a linear combination of $x_ix_j, kx_ix_{n-i} - x_0x_n$ with a non-zero coefficient by $x_0x_n$. After restricting to the intersection of this hypersufrace with the open subset $\{x_n \neq 0\}$, the variable $x_0$ is expressed in terms of the other variables. Take the other $n-1$ equations used to calculate the characteristic number. We can assume that none of them contains $x_0x_n$, so the coefficients of $x_ix_j$ in all conditions are general and there is one linear condition between coefficients of $x_ix_{n-i}$. The matrix of derivatives of these last $n-1$ equations is the sum of the matrix of derivates of the $x_i x_j$-parts and the matrix of derivatives of the $x_ix_{n-i}$-parts. The first matrix generically has full rank and the equations containing $x_i x_j$ are independent from the equations containing $x_i x_{n-1}$, so the sum of these two matrices generically has full rank as well, that is, these hypersufraces are transversal. All of them have degree $2$, so we conclude that
\[
	c_{a_1, \dots, a_n} = 2^{a_2 + \dots + a_{n-1} - 1}. \qedhere
\]
\end{proof}

\subsubsection{The trivial algebra}

We also calculate the characteristic numbers of the trivial algebra, in the sense of Remark \ref{r:generalization}.

\begin{proposition}
\label{e:trivial}
    The characteristic numbers of $\CC[x_0, \dots, x_n]/(x_0, \dots, x_n)^2$ are given by:
    \[
        c_{a_1, \dots, a_n} = 
        \begin{cases}
            2^{a_2} & \text{if } a_1 > 0, a_3 = \dots = a_n = 0 \\
            0 & \text{otherwise} \\
        \end{cases}
    \]
\end{proposition}
\begin{proof}
    The space of matrices associated to $\CC[x_0, \dots, x_n]/(x_0, \dots, x_n)^2$ is
    \[
    \begin{bmatrix}
        x_0 & x_1 & x_2 & \dots & x_{n-1} & x_n \\
        x_1 & 0 & 0 & \dots & 0 & 0 \\
        x_2 & 0 & 0 & \dots & 0 & 0 \\
        \vdots & \vdots & \vdots & \ddots & \vdots & \vdots \\
        x_{n-1} & 0 & 0 & \dots & 0 & 0 \\
        x_{n} & 0 & 0 & \dots & 0 & 0 \\
    \end{bmatrix}. 
    \]
    The non-zero $(1 \times 1)$-minors are just the variables $x_0, x_1, \dots, x_n$, the non-zero $(2 \times 2)$-minors are $x_i x_j$ with $i, j > 0$ and for $k > 2$, all $(k \times k)$-minors are zero. If $a_k \neq 0$ for some $k > 2$ or $a_2 = n$, then $c_{a_1, \dots, a_n} = 0$ for dimensional reasons. Therefore we only need to consider the case where $a_1 + a_2 = n$ and $a_1 > 0$. Take $a_1$ general hyperplane given by linear combinations of $x_0, \dots, x_n$. We may assume that the first equation has a non-zero coefficient by $x_0$ and that the other ones do not contain $x_0$. Thus, after restricting the domain to the first hyperplane, we get a morphism which is the composition of the second Veronese embedding $v_2$ of $\PP^{n-1}$ followed by a linear embedding into a larger projective space (because the space of $(2 \times 2)$-minors is spanned by the set of all degree $2$ monomials in $x_1, \dots x_n$). After restricting to the next hyperplane, we obtain a morphism coming from $v_2$ of $\PP^{n-2}$, and so on. Thus, after intersecting with all $a_1$ hyperplanes, we are left with a morphism coming from $v_2$ of $\PP^{a_2}$. Hence, the intersection number $c_{a_1, a_2, 0, \dots, 0}$ coincides with the degree of the second Veronese embedding of $\PP^{a_2}$, which is $2^{a_2}$.
\end{proof}

\bibliographystyle{alpha}
\bibliography{characteristic-numbers-of-algebras}

\textsc{Jakub Jagiełła, Faculty of Mathematics, Informatics and Mechanics, University of Warsaw}\\
\textit{email address:} jj429138@students.mimuw.edu.pl 

\textsc{Paweł Pielasa, Faculty of Mathematics, Informatics and Mechanics, University of Warsaw}\\
\textit{email address:} p.pielasa@student.uw.edu.pl

\textsc{Anatoli Shatsila, Institute of Mathematics, Jagiellonian University in Krak\'ow, Poland}\\
\textit{email address:} anatoli.shatsila@doctoral.uj.edu.pl.

\end{document}